\documentclass[12pt]{amsart}
\usepackage{cases}
\usepackage{mathrsfs}
\usepackage{color}
\usepackage{latexsym,amssymb}
\usepackage{a4wide}
\usepackage{amsthm}
\usepackage{amsmath}
\usepackage{graphicx}
\input xy
  \xyoption{all}
\usepackage{verbatim}
\usepackage[lite,abbrev,msc-links]{amsrefs}

\numberwithin{equation}{section}
\begin{document}
\theoremstyle{plain}
\newtheorem{thm}{Theorem}[section]
\newtheorem{lem}[thm]{Lemma}
\newtheorem{cor}[thm]{Corollary}
\newtheorem{cor*}[thm]{Corollary*}
\newtheorem{prop}[thm]{Proposition}
\newtheorem{prop*}[thm]{Proposition*}
\newtheorem{conj}[thm]{Conjecture}
\theoremstyle{definition}
\newtheorem{construction}{Construction}
\newtheorem{notations}[thm]{Notations}
\newtheorem{question}[thm]{Question}
\newtheorem{prob}[thm]{Problem}
\newtheorem{rmk}[thm]{Remark}
\newtheorem{remarks}[thm]{Remarks}
\newtheorem{defn}[thm]{Definition}
\newtheorem{claim}[thm]{Claim}
\newtheorem{assumption}[thm]{Assumption}
\newtheorem{assumptions}[thm]{Assumptions}
\newtheorem{properties}[thm]{Properties}
\newtheorem{exmp}[thm]{Example}
\newtheorem{comments}[thm]{Comments}
\newtheorem{blank}[thm]{}
\newtheorem{observation}[thm]{Observation}
\newtheorem{defn-thm}[thm]{Definition-Theorem}
\newtheorem*{Setting}{Setting}

\newcommand{\sA}{\mathscr{A}}
\newcommand{\sB}{\mathscr{B}}
\newcommand{\sC}{\mathscr{C}}
\newcommand{\sD}{\mathscr{D}}
\newcommand{\sE}{\mathscr{E}}
\newcommand{\sF}{\mathscr{F}}
\newcommand{\sG}{\mathscr{G}}
\newcommand{\sH}{\mathscr{H}}
\newcommand{\sI}{\mathscr{I}}
\newcommand{\sJ}{\mathscr{J}}
\newcommand{\sK}{\mathscr{K}}
\newcommand{\sL}{\mathscr{L}}
\newcommand{\sM}{\mathscr{M}}
\newcommand{\sN}{\mathscr{N}}
\newcommand{\sO}{\mathscr{O}}
\newcommand{\sP}{\mathscr{P}}
\newcommand{\sQ}{\mathscr{Q}}
\newcommand{\sR}{\mathscr{R}}
\newcommand{\sS}{\mathscr{S}}
\newcommand{\sT}{\mathscr{T}}
\newcommand{\sU}{\mathscr{U}}
\newcommand{\sV}{\mathscr{V}}
\newcommand{\sW}{\mathscr{W}}
\newcommand{\sX}{\mathscr{X}}
\newcommand{\sY}{\mathscr{Y}}
\newcommand{\sZ}{\mathscr{Z}}
\newcommand{\bZ}{\mathbb{Z}}
\newcommand{\bN}{\mathbb{N}}
\newcommand{\bQ}{\mathbb{Q}}
\newcommand{\bC}{\mathbb{C}}
\newcommand{\bR}{\mathbb{R}}
\newcommand{\bH}{\mathbb{H}}
\newcommand{\bD}{\mathbb{D}}
\newcommand{\bE}{\mathbb{E}}
\newcommand{\bV}{\mathbb{V}}
\newcommand{\cV}{\mathcal{V}}
\newcommand{\cF}{\mathcal{F}}
\newcommand{\bfM}{\mathbf{M}}
\newcommand{\bfN}{\mathbf{N}}
\newcommand{\bfX}{\mathbf{X}}
\newcommand{\bfY}{\mathbf{Y}}
\newcommand{\spec}{\textrm{Spec}}
\newcommand{\dbar}{\bar{\partial}}
\newcommand{\ddbar}{\partial\bar{\partial}}
\newcommand{\redref}{{\color{red}ref}}

\title[$L^2$-Dolbeault resolution of the lowest Hodge piece of a Hodge module] {$L^2$-Dolbeault resolution of the lowest Hodge piece of a Hodge module}

\author[Junchao Shentu]{Junchao Shentu}
\email{stjc@ustc.edu.cn}
\address{School of Mathematical Sciences,
	University of Science and Technology of China, Hefei, 230026, China}
\author[Chen Zhao]{Chen Zhao}
\email{czhao@ustc.edu.cn}
\address{School of Mathematical Sciences,
	University of Science and Technology of China, Hefei, 230026, China}

\begin{abstract}

In this paper, we introduce a coherent subsheaf of Saito's $S$-sheaf, which is a combination of the $S$-sheaf and the multiplier ideal sheaf. We construct its $L^2$-Dolbeault resolution, which generalizes MacPherson's conjecture on the $L^2$ resolution of the Grauert-Riemenschneider sheaf. We also prove various vanishing theorems for the $S$-sheaf (Saito's vanishing theorem, Kawamata-Viehweg vanishing theorem and some new ones like Nadel vanishing theorem) transcendentally. Finally, we discuss some applications of our results on the relative version of Fujita's conjecture (e.g. Kawamata's conjecture).
\end{abstract}

\maketitle

\section{Introduction}
The technique of $L^2$-estimates developed by Andreotti-Vesentini \cite{AV1965} and H\"ormander \cite{Hormander1965} and Saito's theory of Hodge modules \cite{MSaito1988,MSaito1990} have been of great importance in the development of algebraic and complex geometry. The purpose of this paper is to resolve Saito's $S$-sheaf \cite{MSaito1991} (a generalization of the dualizing sheaf) by locally $L^2$-integrable differential forms. We then prove various vanishing theorems for the $S$-sheaf by a transcendental approach. Some of these theorems have previously been proved by Hodge-theoretic approach (e.g. \cite{Suh2018,Wu2017,MSaito1991(2)}) and have applications in the investigation of Shimura varieties \cite{Suh2018}.

Let $X$ be a reduced, irreducible complex space of dimension $n$ and $X^o\subset X_{\rm reg}$ a Zariski open subset. Let $ds^2$ be a hermitian metric on $X^o$ and $\bV:=(\cV,\nabla,\cF^\bullet,h_\bV)$ an $\bR$-polarized variation of Hodge structure on $X^o$, where $h_\bV$ is the Hodge metric defined by $(u,v)_{h_\bV}:=Q(Cu,\overline{v})$ with $Q$ being the polarization of $\bV$ and $C$ the Weil operator. As a generalization of the dualizing sheaf, Saito \cite{MSaito1991} defines the $S$-sheaf $S(IC_X(\bV))$ associated to $\bV$ as the lowest Hodge piece of the intermediate extension $IC_X(\bV)$ and uses it to give a solution to a Koll\'ar's conjecture \cite{Kollar1986}. Saito's $S$-sheaf has been of great importance in the application of the theory of Hodge modules to complex algebraic geometry (see Popa \cite{Popa2018} for a survey).
Let $\varphi:X\to[-\infty,\infty)$ be a quasi-plurisubharmonic (quasi-psh for short) function and $j:X^o\to X$ the open immersion. We introduce the multiplier $S$-sheaf as
\begin{align*}
S(IC_X(\bV),\varphi):=\left\{\alpha\in j_\ast\left(K_{X^o}\otimes S(\bV)\right) \bigg|\int|\alpha|^2_{ds^2,h_\bV}e^{-\varphi}{\rm vol}_{ds^2}<\infty \textrm{  locally at every point }x\in X\right\},
\end{align*}
where $S(\bV):=\cF^{\max\{k|\cF^k\neq0\}}$ is the top indexed nonzero piece of the Hodge filtration $\cF^\bullet$ and $K_{X^o}$ is the holomorphic canonical bundle of $X^o$. This kind of sheaf is a combination of Saito's $S$-sheaf and the multiplier ideal sheaf $\sI(\varphi)$, and has the following features:

\begin{enumerate}
	\item $S(IC_X(\bV),\varphi)$ is a coherent subsheaf of Saito's $S$-sheaf $S(IC_X(\bV))$ (Proposition \ref{prop_S_coherent}). 
	\item $S(IC_X(\bV),\varphi)$ is independent of the choice of $ds^2$ (Proposition \ref{prop_L2_extension_independent}).
	\item $S(IC_X(\bV),0)=S(IC_X(\bV))$ (Theorem \ref{cor_S_phi=0}). $S(IC_X(\bV),-\infty)=0$. This gives an alternative definition of  $S(IC_X(\bV))$ without using the language of Hodge modules. A similar relation regarding $(0,0)$-forms is also recently observed by  Schnell-Yang \cite{SY2023}.
	\item $S(IC_X(\bV),\varphi)$ has the functorial property (Proposition \ref{prop_L2ext_birational}).
	\item 
Assume that $X$ is smooth, $X\backslash X^o\subset X$ is a (possibly empty) normal crossing divisor and $\bV$ is an $\bR$-polarized variation of Hodge structure with unipotent local monodromies. Let $\varphi$ be a plurisubharmonic (psh for short) function with generalized analytic singularities along $X\backslash X^o$ (Definition \ref{defn_generalized_analytic_singularity}). Then, as shown in Proposition \ref{prop_S_vs_mulidealsheaf}, there is an isomorphism
\begin{align}\label{align_S_multiideal}
S(IC_X(\bV),\varphi)\simeq S(IC_X(\bV))\otimes \sI(\varphi).
\end{align}
When the local monodromies are not necessarily unipotent, (\ref{align_S_multiideal}) only holds on $X\backslash X^o$. Nevertheless, as stated in Proposition \ref{prop_key_est}, there is a decomposition
\begin{align}\label{align_S_vs_multiideal}
S(IC_X(\bV),\varphi)\simeq\bigoplus_{i=1}^m\sI(\varphi_i)\otimes\omega_{X}
\end{align} 
locally at every point of $X\backslash X^o$. Here $\varphi_i$ are quasi-psh functions which depend on $\varphi$ and the eigenvalues of the local monodromies of $\bV$. (\ref{align_S_vs_multiideal}) plays a crucial role in the study of the relative Fujita conjecture (\S \ref{section_geo_app}).

	\item $S(IC_X(\bV),\varphi)$ satisfies an Ohsawa-Takegoshi extension theorem, at least when $\varphi$ is generically smooth (Theorem \ref{thm_S_OT_extension}). 
\end{enumerate}
Let $(E,h_\varphi)$ be a holomorphic vector bundle on $X$ with a possibly singular hermitian metric $h_\varphi$. Throughout this paper, by a singular hermitian metric $h_\varphi=e^{-\varphi}h_0$ we always assume that $h_0$ is a smooth hermitian metric and $\varphi$ is a quasi-psh function.  $S(IC_X(\bV),\varphi)$ is independent of the choice of the decomposition $h_\varphi=e^{-\varphi}h_0$ (Lemma \ref{lem_S_change_varphi_bounded}). Let  $\sD^{n,q}_{X,ds^2}(S(\bV)\otimes E,h_\bV\otimes h_\varphi)$ denote the sheaf of measurable $S(\bV)\otimes E$-valued $(n,q)$-forms $\alpha$ such that $\alpha$ and $\dbar\alpha$ are locally square integrable with respect to $ds^2$ and $h_\bV\otimes h_\varphi$. Let $\omega$ be the $(1,1)$-form associated to $ds^2$.
The main result of the present paper is
\begin{thm}[=Theorem \ref{thm_main_local1}]\label{thm_main_resolution}
	Assume that locally at every point $x\in X$ there is a neighborhood $U$ of $x$, a strictly psh function $\lambda\in C^2(U)$ and a bounded psh function $\Phi\in C^2(U\cap X^o)$ such that $\sqrt{-1}\ddbar\lambda\lesssim\omega|_{U\cap X^o}\lesssim\sqrt{-1}\ddbar\Phi$.
	Then the complex of sheaves
	\begin{align*}
	0\to S(IC_X(\bV),\varphi)\otimes E\to \sD^{n,0}_{X,ds^2}(S(\bV)\otimes E,h_\bV\otimes h_\varphi)\stackrel{\dbar}{\to}
	\cdots\stackrel{\dbar}{\to} \sD^{n,n}_{X,ds^2}(S(\bV)\otimes E,h_\bV\otimes h_\varphi)\to0
	\end{align*}
	is exact. If $X$ is moreover compact, then there is an isomorphism
	\begin{align}\label{align_MacPherson_Hodge_module}
	H^q(X,S(IC_X(\bV),\varphi)\otimes E)\simeq H^{n,q}_{(2),\rm max}(X^o, S(\bV)\otimes E;ds^2,h_\bV\otimes h_\varphi),\quad\forall q.
	\end{align}
\end{thm}
When $ds^2$ is the hermitian metric on $X$ (Definition \ref{defn_hermitian_metric}), $\bV=\bC_{X_{\rm reg}}$, $\varphi=0$ and $E=\sO_X$ is endowed with the trivial metric, (\ref{align_MacPherson_Hodge_module}) implies the results by Pardon-Stern \cite{Pardon_Stern1991} and by Ruppenthal \cite{Ruppenthal2014} on MacPherson's conjecture. 
\subsection{Vanishing theorems}
The $L^2$-resolution of the multiplier $S$-sheaf allows us to investigate the $S$-sheaf by means of analytical methods. Theorem \ref{thm_main_resolution} is used to give a transcendental prove to Koll\'ar's conjecture (\cite[\S 5]{Kollar1986}) on the derived pushforward of $S(IC_X(\bV))$ in \cite{SC2021_kollar}, as well as its generalizations. In the present paper we deduce from Theorem \ref{thm_main_resolution} various vanishing theorems for Saito's $S$-sheaf. 
\begin{thm}[Nadel type vanishing theorem, =Corollary \ref{cor_Nadel_rel}]\label{thm_R_Nadel_vanishing}
	Let $f:X\to Y$ be a surjective proper K\"ahler holomorphic map between irreducible complex spaces. Let $\bV$ be an $\bR$-polarized variation of Hodge structure defined on a Zariski open subset of $X$. Let $(L,h_\varphi)$ a holomorphic line bundle on $X$ with a possibly singular hermitian metric $h_\varphi:=e^{-\varphi}h$. Assume that $\sqrt{-1}\Theta_{h_\varphi}(L)$ is $f$-positive. Then
	\begin{align*}
	R^if_\ast(S(IC_X(\bV),\varphi)\otimes L)=0,\quad \forall i>0.
	\end{align*}
\end{thm}
When $X$ is smooth, $Y$ is a point and $\bV=\bC_X$, we recover the Nadel vanishing theorem \cite{Nadel1990}. Many interesting generalizations are obtained, such as \cite{Demailly1982,Matsumura2014,Matsumura2015,Iwai2021}. 
When $X$ is a projective variety, with a careful choice of $h_\varphi$ we obtain 
\begin{cor}[Demailly-Kawamata-Viehweg type vanishing theorem, =Corollary \ref{cor_KV_vanishing}]\label{cor_main_algebraic}
	Let $X$ be a projective algebraic variety of dimension $n$ and $\bV$ an $\bR$-polarized variation of Hodge structure defined on a Zariski open subset of $X$. Let $L$ be a line bundle such that some positive multiple $mL=F+D$ where $F$ is a nef line bundle and $D$ is an effective divisor. Then
	$$H^q(X,S(IC_X(\bV),\frac{\varphi_D}{m})\otimes L)=0,\quad \forall q>n-{\rm nd}(L).$$
	Here $\varphi_D$ is the psh function associated to $D$.
\end{cor}
When $L$ is nef and big, this vanishing theorem has been  established by Suh \cite{Suh2018} and Wu \cite{Wu2017} by means of Hodge theoretic methods, which generalizes Saito's vanishing theorem for the $S$-sheaf  \cite{MSaito1991(2)}. When $X$ is smooth and $\bV=\bC_X$ is the trivial Hodge module, it reduces to the Demailly-Kawamata-Viehweg vanishing theorem \cite[6.25]{Demailly2012}, with its roots traced back to Kawamata and Viehweg in \cite{Kawamata1982,Viehweg1982}. Recent developments include \cite{Cao2014,Wu2022,Inayama2022,DP2003,Demailly1991}.
The Kodaira-Nakano-Kazama  vanishing theorem, the relative vanishing theorem  and Fujino-Enoki-Koll\'ar  injectivity theorem are generalized to coefficients in $S(IC_X(\bV),\varphi)$ (Theorem \ref{thm_partial_vanishing}, Corollary \ref{cor_Nadel_rel}, Theorem \ref{thm_Enoki_Kollar_inj}). It also implies the Esnault-Viehweg type injectivity theorem (Corollary \ref{cor_wulei}) which has been proved in \cite[Theorem 1.4]{Wu2017}  using Hodge theoretic methods.  For the case that $X$ is smooth and $\bV$ is trivial, these vanishing theorems are mainly due to the efforts of Kodaira \cite{Kodaira1953}, Nakano \cite{Nakano1974}, Kazama \cite{Kazama1973}, Takegoshi \cite{Takegoshi1985}, Koll\'ar \cite{Kollar1986}, Enoki \cite{Enoki1993}, Fujino \cite{Fujino2017} and Cao-P\v{a}un \cite{Paun2020}.
Readers may also refer to the works of Fujino and Matsumura \cite{Matsumura2018,Matsumura20182,Matsumura20183,Matsumura20184} and the references therein.

\subsection{Application to the relative version of Fujita's conjecture}\label{section_geo_app}
As a relative version of Fujita's conjecture \cite{Fujita1987}, Kawamata raised the following conjecture in \cite{Kawamata2002} with the case $\dim_\bC X\leq 4$ settled therein.
\begin{conj}[Kawamata \cite{Kawamata2002}]\label{conj_kawamata}
	Let $f:Y\to X$ be a proper morphism between smooth projective algebraic varieties. Assume that the degenerate loci of $f$ is contained in a normal crossing divisor $D\subset X$. Let $L$ be an ample line bundle on $X$. Then $R^qf_\ast\omega_Y\otimes L^{\dim X+1}$ is generated by global sections for every $0\leq q\leq \dim_\bC X$.
\end{conj}
With the help of the $L^2$-Dolbeault resolution (Theorem \ref{thm_main_resolution}) on $R^qf_\ast\omega_Y$ (an example of Saito's $S$-sheaf), we are able to investigate the separation of jets of $R^qf_\ast\omega_Y\otimes L^{\dim X+1}$ using the transendental method developed by Angehrn-Siu \cite{Siu1995} and Demailly \cite[Theorem 7.4]{Demailly2012}.
\begin{cor}\label{cor_Siu}
	Let $f:Y\to X$ be a proper holomorphic morphism from a K\"ahler manifold to a projective algebraic variety where $\dim_{\bC}X=n$. Assume that the degenerate loci of $f$ is contained in a normal crossing divisor $D\subset X$. Let $L$ be an ample line bundle on $X$. Assume that there is a positive number $\kappa>0$ such that $$L^k\cdot W\geq\left(\frac{1}{2}n(n+2r-1)+\kappa\right)^d$$
	for any irreducible subvariety $W$ of dimension $0\leq d\leq n$ in $X$. Let $0\leq q\leq n$. Then the global holomorphic sections of $R^qf_\ast\omega_Y\otimes L$ separate any set of $r$ distinct points $x_1,\dots, x_r\in X$, i.e. there is a surjective map
	$$H^0(X,R^qf_\ast\omega_Y\otimes L)\to \bigoplus_{1\leq k\leq r}R^qf_\ast\omega_Y\otimes L\otimes \sO_{X,x_k}/m_{X,x_k}.$$
\end{cor}
A similar result is obtained by Wu \cite{Wu2017} Hodge theoretically.  When $f={\rm Id}$, this reduces to the result in \cite{Siu1995}.

{By using Demailly's singular metric on the adjoint bundles (\cite[Theorem 7.4]{Demailly2012}) we are able obtain the relative version of \cite[Theorem 7.4]{Demailly2012}.}
\begin{cor}\label{cor_Demailly}
	Let $f:Y\to X$ be a proper holomorphic morphism from a K\"ahler manifold to a projective algebraic variety where $\dim_{\bC}X=n$. Assume that the degenerate loci of $f$ is contained in a normal crossing divisor $D\subset X$. Let $L$ be an ample line bundle and $G$ a nef line bundle on $X$. Then there is a surjective map
	$$H^0(X,R^qf_\ast\omega_Y\otimes \omega_X\otimes L^{\otimes m}\otimes G)\to \bigoplus_{1\leq k\leq r}R^qf_\ast\omega_Y\otimes \omega_X\otimes L^{\otimes m}\otimes G\otimes \sO_{X,x_k}/m^{s_k+1}_{X,x_k}$$
	at arbitrary points $x_1,\dots, x_r\in X$  for every $0\leq q\leq \dim_\bC Y-\dim_\bC X$,
	provided that $m\geq 2+\sum_{1\leq k\leq r}\binom{3n+2s_k-1}{n}$. 
\end{cor}
Since Theorem \ref{thm_R_Nadel_vanishing} holds for a general Hodge module, we actually prove the analogues of Corollary \ref{cor_Siu} and Corollary \ref{cor_Demailly} for Hodge modules (Theorem \ref{thm_Hodge_Siu} and Theorem \ref{thm_sepjets}). We also obtain the results on separating jets with an optimal bound (Theorem \ref{thm_relFujita_semiample}) when $L$ is ample and base point free.

\begin{rmk}
	For a general Hodge module $M$, the semi-simplicity of $M$ provides a unique decomposition $M=\bigoplus IC_{Z_i}(\bV_i)$ where  $IC_{Z_i}(\bV_i)$ is a  Hodge module with its strict support an irreducible Zariski closed subset $Z_i\subset X$ for each $i$. Then $S(M,\varphi)$ could be defined as $\bigoplus S(IC_{Z_i}(\bV_i),\varphi|_{Z_i})$. The main results in the present paper hold for a general Hodge module as long as they are valid for Hodge modules with strict support. Therefore, we only consider Hodge modules with strict support in the present paper.
\end{rmk}
\textbf{Acknowledgment:} Both authors would like to thank Zhenqian Li, Ya Deng and Ruijie Yang for many helpful conversations. The first author also thanks Lei Zhang for his interest in this work.

{\bf Conventions and Notations:} 
\begin{itemize}
	\item All complex spaces are separated, irreducible, reduced, paracompact and countable at infinity. Let $X$ be a complex space. A Zariski closed subset (=closed analytic subset) $Z\subset X$ is a closed subset which is locally defined as the zeros of a set of holomorphic functions. A subset $Y\subset X$ is called Zariski open if $X\backslash Z\subset X$ is Zariski closed.
	\item Let $Y$ be a complex manifold and $E$ a hermitian vector bundle on $Y$. Let $\Theta\in A^{1,1}(Y,End(E))$. Denote $\Theta\geq0$ if $\Theta$ is Nakano semipositive.
	Let $\Theta_1,\Theta_2\in A^{1,1}(Y,End(E))$. Then $\Theta_1\geq\Theta_2$ stands for $\Theta_1-\Theta_2\geq0$.
	\item A $C^\infty$ form on a complex space $X$ is a $C^\infty$ form $\alpha$ on $X_{\rm reg}$ so that the following statement hold: Locally at every point $x\in X$ there is an open neighborhood $U$ of $x$, a holomorphic embedding $\iota:U\to \bC^N$ and $\beta\in C^\infty(\bC^N)$ such that $\iota^\ast\beta=\alpha$ on $U\cap X_{\rm reg}$.
	\item A psh (resp. strictly psh) function on a complex space $X$ is a function $\lambda:X\to[-\infty,\infty)$ such that locally at every point $x\in X$ there is a neighborhood $U$ of $x$, a closed immersion $\iota:U\to\Omega$ into a holomorphic manifold $\Omega$ and a  psh (resp. strictly psh) function $\Lambda$ on $\Omega$ such that $\iota^\ast\Lambda=\lambda$. A function $\varphi$ on $X$ is called quasi-psh if it can	be written locally as a sum $\varphi=\alpha+\psi$ of a $C^\infty$ function $\alpha$ and a psh  function $\psi$.
	
	\item Let $\varphi$ be a quasi-psh function on a holomorphic manifold $X$. $\sI(\varphi)\subset\sO_X$ denotes the multiplier ideal sheaf consisting of holomorphic functions $f$ such that $|f|^2e^{-\varphi}$ is locally integrable.
	\item Let $(Y,ds^2)$ be a hermitian manifold and $E$ a holomorphic vector bundle on $Y$. A singular hermitian metric on $E$ is a measurable section $h\in E^\ast\otimes \overline{E}^\ast$ such that $h=e^{-\varphi}h_0$ for some smooth hermitian metric $h_0$ and some quasi-psh function $\varphi$. 
	\item Let $\alpha$, $\beta$ be functions (resp. metrics or $(1,1)$-forms). We denote $\alpha\lesssim\beta$ if $\alpha\leq C\beta$ for some constant $C>0$. Denote $\alpha\sim\beta$ if both $\alpha\lesssim\beta$ and $\beta\lesssim\alpha$ hold.
\end{itemize}
\section{Preliminaries on Saito's $S$-sheaf}\label{section_Hodge_module}
\subsection{Saito's $S$-sheaf}\label{subsection_Hodge_module}
Readers may see \cite{MSaito1988,MSaito1990,MSaito1991(2),Schnell_introMHS,Peter_Steenbrink2008} for the theory of Hodge module. In the present paper we will not use the theory of Hodge module. Instead, a concrete construction of Saito's $S$-sheaf using Deligne's extension will be used. This construction is originated by Koll\'ar in \cite{Kollar1986}.

Let $X$ be a complex space and $X^o\subset X_{\rm reg}$ a Zariski open subset. Let $\bV:=(\cV,\nabla,\cF^\bullet,h_\bV)$ be an $\bR$-polarized variation of Hodge structure (\cite[\S 1]{Cattani_Kaplan_Schmid1986}) on $X^o$.
The $S$-sheaf $S(IC_X(\bV))$ associated with $\bV$ is defined as follows.
\begin{enumerate}
	\item (Log smooth case): Assume that $X$ is smooth and $E:=X\backslash X^o$ is a simple normal crossing divisor. Let $E=\cup E_i$ be the irreducible decomposition. By \cite[\S II, Proposition 5.4]{Deligne1970}, there is a logarithmic flat holomorphic vector bundle  $(\widetilde{\cV}_{-1},\widetilde{\nabla})$ (unique up to isomorphisms):
	\begin{align*}
	\widetilde{\nabla}: \widetilde{\cV}_{-1}\to\Omega_{X}(\log E)\otimes\widetilde{\cV}_{-1},
	\end{align*}
	such that $(\widetilde{\cV}_{-1},\widetilde{\nabla})|_{X^o}$ is holomorphically equivalent to $(\cV,\nabla)$ and the eigenvalues of the residue operator
	\begin{align*}
	{\rm Res}_{E_i}\widetilde{\nabla}:\widetilde{\cV}_{-1}|_{E_i}\to \widetilde{\cV}_{-1}|_{E_i}
	\end{align*}
	lie in $(-1,0]$. 
	Let $j:X^o\to X$ denote the open immersion. By \cite[Theorem 1.1]{MSaito1991}, the $S$-sheaf can be described as $$S(IC_{X}(\bV))=R(IC_{X}(\bV))\otimes\omega_{X}$$ where
	\begin{align*}
	R(IC_{X}(\bV))=j_\ast(S(\bV))\cap\widetilde{\cV}_{-1},\quad S(\bV):=\cF^{\max\{k|\cF^k\neq0\}}.
	\end{align*}
	Moreover, $R(IC_{X}(\bV))$ is a holomorphic subbundle of $\widetilde{\cV}_{-1}$ according to the nilpotent orbit theorem (Schmid \cite{Schmid1973} and Cattani-Kaplan-Schmid \cite{Cattani_Kaplan_Schmid1986}).
	\item (General case): Let $\pi:\widetilde{X}\to X$ be a proper bimeromorphic morphism such that $\pi$ is biholomorphic over $X^o$ and the exceptional loci $E:=\pi^{-1}(X\backslash X^o)$ is a simple normal crossing divisor. One defines
	\begin{align*}
	S(IC_X(\bV)):=\pi_\ast\left(S(IC_{\widetilde{X}}(\pi^\ast\bV))\right)\simeq R\pi_\ast\left(S(IC_{\widetilde{X}}(\pi^\ast\bV))\right).
	\end{align*}
	Saito \cite{MSaito1991} shows that $S(IC_X(\bV))$ is independent of the choice of the desingularization $\pi$. We will provide another proof of this fact by characterizing $S(IC_X(\bV))$ using $L^2$ holomorphic sections (Corollary \ref{cor_S_phi=0}).
\end{enumerate}
The existence of the $S$-sheaf (associated to $\bV$) was conjectured by  Koll\'ar \cite{Kollar1986}, as a generalization of the dualizing sheaf,  to admit a package of theorems such as Koll\'ar's vanishing theorem, torsion freeness and the decomposition theorem. The construction of $S$-sheaf and its package of theorems are settled by Saito in \cite{MSaito1991} through his theory of Hodge modules.
\subsection{Geometric behavior of the Hodge metric}\label{section_norm_est_VHS}
Let $\bV=(\cV,\nabla,\cF^\bullet,h_\bV)$ be an $\bR$-polarized variation of Hodge structure on $(\Delta^\ast)^n\times \Delta^{n'}$. Let $s_1,\dots,s_n$ be holomorphic coordinates on $(\Delta^\ast)^n$ and denote $D_i=\{s_i=0\}\subset\Delta^{n+n'}$. Let $N_i$ be the unipotent part of ${\rm Res}_{D_i}\nabla$ and let 
$$p:\bH^{n}\times \Delta^{n'}\to (\Delta^\ast)^n\times \Delta^{n'},$$ 
$$(z_1,\dots,z_n,w_1,\dots,w_{n'})\mapsto(e^{2\pi\sqrt{-1}z_1},\dots,e^{2\pi\sqrt{-1}z_n},w_1,\dots,w_{n'})$$
be the universal covering. Let
$W^{(1)}=W(N_1),\dots,W^{(n)}=W(N_1+\cdots+N_n)$ be the monodromy weight filtrations on $V:=\Gamma(\bH^n\times \Delta^{n'},p^\ast\cV)^{p^\ast\nabla}$.
The following important norm estimate for flat sections is proved by Cattani-Kaplan-Schmid in \cite[Theorem 5.21]{Cattani_Kaplan_Schmid1986} for the case when $\bV$ has quasi-unipotent local monodromy and by Mochizuki in \cite[Part 3, Chapter 13]{Mochizuki20072} for the general case.

\begin{thm}\label{thm_Hodge_metric_asymptotic}
	For any $0\neq v\in {\rm Gr}_{l_n}^{W^{(n)}}\cdots{\rm Gr}_{l_1}^{W^{(1)}}V$, one has
	\begin{align*}
	|v|_{h_\bV}\sim \left(\frac{\log|s_1|}{\log|s_2|}\right)^{l_1}\cdots\left(-\log|s_n|\right)^{l_n}
	\end{align*}
	over any region of the form
	$$\left\{(s_1,\dots s_n,w_1,\dots,w_{n'})\in (\Delta^\ast)^n\times \Delta^{n'}\bigg|\frac{\log|s_1|}{\log|s_2|}>\epsilon,\dots,-\log|s_n|>\epsilon,(w_1,\dots,w_{n'})\in K\right\}$$
	for any $\epsilon>0$ and an arbitrary compact subset $K\subset \Delta^{n'}$ .
\end{thm}
\begin{lem}\label{lem_W_F}
	Assume that $n=1$. Then $W_{-1}(N_1)\cap \big(j_\ast S(\bV)\cap\cV_{-1}\big)|_{\bf 0}=0$.
\end{lem}
\begin{proof}
	Assume that $W_{-1}(N_1)\cap \big(j_\ast S(\bV)\cap\cV_{-1}\big)|_{\bf 0}\neq0$ and let $k$ be the weight of $\bV$. Let $l=\max\{l|W_{-l}(N_1)\cap \big(j_\ast S(\bV)\cap\cV_{-1}\big)|_{\bf 0}\neq0\}$. Then $l\geq 1$. 
	By \cite[6.16]{Schmid1973}, the filtration $j_\ast \cF^\bullet\cap\cV_{-1}$ induces a pure Hodge structure of weight $m+k$ on $W_{m}(N_1)/W_{m-1}(N_1)$. Moreover, 
	\begin{align}\label{align_hard_lef_N}
	N^l: W_{l}(N_1)/W_{l-1}(N_1)\to W_{-l}(N_1)/W_{-l-1}(N_1)
	\end{align}
	is an isomorphism of type $(-l,-l)$. Denote $S(\bV)=\cF^p$. By the definition of $l$, any nonzero element $\alpha\in W_{-l}(N_1)\cap \big(j_\ast S(\bV)\cap\cV_{-1}\big)|_{\bf 0}$ induces a nonzero $[\alpha]\in W_{-l}(N_1)/W_{-l-1}(N_1)$ of Hodge type $(p,k-l-p)$. Since (\ref{align_hard_lef_N}) is an isomorphism, there is $\beta\in W_{l}(N_1)/W_{l-1}(N_1)$ of Hodge type $(p+l,k-p)$ such that $N^l(\beta)=[\alpha]$. However, $\beta=0$ since $\cF^{p+l}=0$. This contradicts to the fact that $[\alpha]\neq0$. Consequently, $W_{-1}(N_1)\cap \big(j_\ast S(\bV)\cap\cV_{-1}\big)|_{\bf 0}$ must be zero.
\end{proof}

The following Nakano semi-positivity property of the curvature of $S(\bV)$ enables us to apply H\"ormander's estimate to $S(\bV)$.
\begin{thm}{\cite[Lemma 7.18]{Schmid1973}}\label{thm_geq0_S(V)}
	Let $\bV=(\cV,\nabla,\cF^\bullet,h_\bV)$ be an $\bR$-polarized variation of Hodge structure over a complex manifold. Then $\sqrt{-1}\Theta_{h_{\bV}}(S(\bV))\geq 0$.
\end{thm}
\subsection{$L^2$-adapted local frame}\label{section_L2_adapted_frame}
Let $X=\Delta^n\times \Delta^{n'}$, $X^o=(\Delta^\ast)^n\times \Delta^{n'}$ and let  $j:X^o\to X$ be the open immersion. Denote by $z_1,\dots,z_n$ the coordinates on $\Delta^n$ and by $w_1,\dots,w_{n'}$ the coordinates on $\Delta^{n'}$. Let $D_i:=\{z_i=0\}\subset X$, $i=1,\dots, n$. Let $\bV=(\cV,\nabla,\cF^\bullet,h)$ be an $\bR$-polarized variation of Hodge structure on $X^o$. The aim of this subsection is to give the norm estimate of the specific frame of $R(IC_X(\bV))$ at the origin ${\bf 0}=(0,\dots,0)$ which is introduced by Deligne \cite{Deligne1970}. 

Let 
$$p:\bH^{n}\times \Delta^{n'}\to (\Delta^\ast)^n\times \Delta^{n'},$$ 
$$(z_1,\dots,z_n,w_1,\dots,w_{n'})\mapsto(e^{2\pi\sqrt{-1}z_1},\dots,e^{2\pi\sqrt{-1}z_n},w_1,\dots,w_{n'})$$
be the universal covering.
For each $i=1,\dots,n$, let $T_i$ be the monodromy operators along $D_i$ and $N_i$  the unipotent part of ${\rm Res}_{D_i}(\nabla)$.
Since $T_1,\dots,T_n$ are pairwise commutative, there is a finite decomposition 
$$\cV_{-1}|_{\bf 0}=\bigoplus_{-1<\alpha_1,\dots,\alpha_n\leq 0}\bV_{\alpha_1,\dots,\alpha_n}$$
such that $(T_i-e^{2\pi\sqrt{-1}\alpha_i}{\rm Id})$ is unipotent on $\bV_{\alpha_1,\dots,\alpha_n}$ for each $i=1,\dots,n$. 
Let $$v_1,\dots, v_N\in (\cV_{-1}\cap j_\ast S(\bV))|_{\bf 0}\cap\bigcup_{-1<\alpha_1,\dots,\alpha_n\leq 0}\bV_{\alpha_1,\dots,\alpha_n}$$
be an orthogonal basis of $(\cV_{-1}\cap j_\ast S(\bV))|_{\bf 0}\simeq \Gamma(\bH^n\times\Delta^{n'},p^\ast S(\bV))^{p^\ast\nabla}$. Then $\widetilde{v_1},\dots,\widetilde{v_N}$ that are determined by
\begin{align}\label{align_adapted_frame}
\widetilde{v_j}:={\rm exp}\left(\sum_{i=1}^n\log z_i(\alpha_i{\rm Id}+N_i)\right)v_j\textrm{ if } v_j\in\bV_{\alpha_1,\dots, \alpha_n},\quad \forall j=1,\dots,N
\end{align}
form a frame of $\cV_{-1}\cap j_\ast S(\bV)$.
To be precise, we always use the notation $\alpha_{D_i}(\widetilde{v_j})$ instead of $\alpha_i$ in (\ref{align_adapted_frame}). By (\ref{align_adapted_frame}) we see that 
\begin{align*}
|\widetilde{v_j}|^2_{h}&\sim\left|\prod_{i=1}^nz_i^{\alpha_{D_i}(\widetilde{v_j})}{\rm exp}\left(\sum_{i=1}^nN_i\log z_i\right)v_j\right|^2_{h}\\\nonumber
&\sim|v_j|^2_{h}\prod_{i=1}^n |z_i|^{2\alpha_{D_i}(\widetilde{v_j})},\quad j=1,\dots,N
\end{align*}
where $\alpha_{D_i}(\widetilde{v_j})\in(-1,0]$, $\forall i=1,\dots, n$. 
By Theorem \ref{thm_Hodge_metric_asymptotic} and Lemma \ref{lem_W_F}
one has
\begin{align*}
|v_j|^2_{h}\sim \left(\frac{\log|s_1|}{\log|s_2|}\right)^{l_1}\cdots\left(-\log|s_n|\right)^{l_n},\quad l_1\leq l_2\leq\dots\leq l_{n},
\end{align*}
over any region of the form
$$\left\{(s_1,\dots, s_n,w_1,\dots,w_{m})\in (\Delta^\ast)^n\times \Delta^{m}\bigg|\frac{\log|s_1|}{\log|s_2|}>\epsilon,\dots,-\log|s_n|>\epsilon,(w_1,\dots,w_{m})\in K\right\}$$
for any $\epsilon>0$ and an arbitrary compact subset $K\subset \Delta^{m}$. Hence we know that
\begin{align*}
1\lesssim |v_j|\lesssim|z_1\cdots z_n|^{-\epsilon},\quad\forall\epsilon>0.
\end{align*} 

The local frame $(\widetilde{v_1},\dots,\widetilde{v_N})$ is $L^2$-adapted in the following sense.
\begin{defn}(Zucker \cite[page 433]{Zucker1979})
	Let $(E,h)$ be a vector bundle with a possibly singular hermitian metric $h$ on a hermitian manifold $(X,ds^2_0)$. A holomorphic local frame $(v_1,\dots,v_N)$ of $E$ is called $L^2$-adapted if, for every set of measurable functions $\{f_1,\dots,f_N\}$, $\sum_{i=1}^Nf_iv_i$ is locally square integrable if and only if $f_iv_i$ is locally square integrable for each $i=1,\dots,N$.
\end{defn}
To see that $(\widetilde{v_1},\dots,\widetilde{v_N})$ is $L^2$-adapted, let us consider the measurable functions $f_1,\dots,f_N$. If 
$$\sum_{j=1}^N f_j\widetilde{v_j}={\rm exp}\left(\sum_{i=1}^nN_i\log z_i\right)\left(\sum_{j=1}^N f_j\prod_{i=1}^n |z_i|^{\alpha_{D_i}(\widetilde{v_j})}v_j\right)$$
is locally square integrable, then 
$$\sum_{j=1}^N f_j\prod_{i=1}^n |z_i|^{\alpha_{D_i}(\widetilde{v_j})}v_j$$
is locally square integrable because the entries of the matrix ${\rm exp}\left(-\sum_{i=1}^nN_i\log z_i\right)$ are $L^\infty$-bounded.
Since $(v_1,\dots,v_N)$ is an orthogonal basis, 
$|f_j\widetilde{v_j}|_{h}\sim\prod_{i=1}^n |z_i|^{\alpha_{D_i}(\widetilde{v_j})}|f_jv_j|_{h}$ is locally square integrable for each $j=1,\dots,N$. 

In conclusion, we obtain the following proposition.
\begin{prop}\label{prop_adapted_frame}
	Let $(X,ds^2_0)$ be a hermitian manifold and $D$ a normal crossing divisor on $X$. Let $\bV$ be an $\bR$-polarized variation of Hodge structure on $X^o:=X\backslash D$. Then there is an $L^2$-adapted holomorphic local frame $(\widetilde{v_1},\dots,\widetilde{v_N})$ of $\cV_{-1}\cap j_\ast S(\bV)$ at every point $x\in D$. 
	Let $z_1,\cdots,z_n$ be holomorphic local coordinates on $X$ so that $D
	=\{z_1\cdots z_r=0\}$. Then there are constants  $\alpha_{D_i}(\widetilde{v_j})\in(-1,0]$, $i=1,\dots, r$, $j=1,\dots,N$ and positive real functions $\lambda_j\in C^\infty(X\backslash D)$, $j=1,\dots,N$ such that
	\begin{align*}
	|\widetilde{v_j}|^2\sim\lambda_j\prod_{i=1}^r |z_i|^{2\alpha_{D_i}(\widetilde{v_j})},\quad \forall j=1,\dots,N
	\end{align*}
	and
	$$1\lesssim \lambda_j\lesssim|z_1\cdots z_r|^{-\epsilon},\quad\forall\epsilon>0,\quad \forall j=1,\dots,N$$
\end{prop}
\section{Preliminary on $L^2$-cohomology}
\subsection{$L^2$-Dolbeault cohomology and $L^2$-Dolbeault complex}
Let $(Y,ds^2)$ be a hermitian manifold of dimension $n$ and $(E,h)$ a holomorphic vector bundle on $Y$ with a possibly singular hermitian metric.
Let $\sA^{p,q}_Y$ denote the sheaf of $C^\infty$ $(p,q)$-forms on $Y$ for every $0\leq p,q\leq n$. Let $L^{p,q}_{(2)}(Y,E;ds^2,h)$ be the space of square integrable $E$-valued $(p,q)$-forms on $Y$ with respect to the metrics $ds^2$ and $h$. Denote $\dbar_{\rm max}$ to be the maximal extension of the $\dbar$ operator defined on the domains
$$D^{p,q}_{\rm max}(Y,E;ds^2,h):=\textrm{Dom}^{p,q}(\dbar_{\rm max})=\{\phi\in L_{(2)}^{p,q}(Y,E;ds^2,h)|\dbar\phi\in L_{(2)}^{p,q+1}(Y,E;ds^2,h)\}.$$
Here $\dbar$ is taken in the sense of distribution.
The $L^2$ cohomology $H_{(2),\rm max}^{p,\bullet}(Y,E;ds^2,h)$ is defined as the cohomology of the complex
\begin{align*}
D^{p,\bullet}_{\rm max}(Y,E;ds^2,h):=D^{p,0}_{\rm max}(Y,E;ds^2,h)\stackrel{\dbar_{\rm max}}{\to}\cdots\stackrel{\dbar_{\rm max}}{\to}D^{p,n}_{\rm max}(Y,E;ds^2,h).
\end{align*}
Let $X$ be a complex space and $X^o\subset X_{\rm reg}$ a Zariski open subset of the regular locus $X_{\rm reg}$. Let $ds^2$ be a hermitian metric on $X^o$ and $(E,h)$ a holomorphic vector bundle on $X$ with a possibly singular metric.
Let $U\subset X$ be an open subset. Define $L_{X,ds^2}^{p,q}(E,h)(U)$ to be the space of measurable $E$-valued $(p,q)$-forms $\alpha$ on $U\cap X^o$ such that for every point $x\in U$, there is a neighborhood $ V_x$ of $x$ so that 
$$\int_{V_x\cap X^o}|\alpha|^2_{ds^2,h}{\rm vol}_{ds^2}<\infty.$$
For each $p$ and $q$, we define a sheaf $\sD_{X,ds^2}^{p,q}(E,h)$ on $X$ by
$$\sD_{X,ds^2}^{p,q}(E,h)(U):=\{\phi\in L_{X,ds^2}^{p,q}(E,h)(U)|\bar{\partial}_{\rm max}\phi\in L_{X,ds^2}^{p,q+1}(E,h)(U)\}$$
for every open subset $U\subset X$.
Define the $L^2$-Dolbeault complex of sheaves $\sD_{X,ds^2}^{p,\bullet}(E,h)$ as
\begin{align*}
\sD_{X,ds^2}^{p,0}(E,h)\stackrel{\dbar}{\to}\sD_{X,ds^2}^{p,1}(E,h)\stackrel{\dbar}{\to}\cdots\stackrel{\dbar}{\to}\sD_{X,ds^2}^{p,n}(E,h)
\end{align*}
where $\dbar$ is defined in the sense of distribution.

\begin{defn}\label{defn_hermitian_metric}
	Let $X$ be a complex space and $ds^2$ a hermitian metric on $X_{\rm reg}$. We say that $ds^2$ is a hermitian metric on $X$ if, for every $x\in X$, there is a neighborhood $U$ of $x$ and a holomorphic closed immersion $U\subset V$ into a complex manifold such that $ds^2|_U\sim ds^2_V|_{U}$ for some hermitian metric $ds^2_V$ on $V$. If $ds^2|_{X_{\rm reg}}$ is moreover a K\"ahler metric, we say that $ds^2$ is a K\"ahler hermitian metric.
\end{defn}

\begin{lem}\label{lem_fine_sheaf}
	Let $X$ be a complex space and $X^o\subset X_{\rm reg}$ a Zariski open subset. Let $ds^2$ be a hermitian metric on $X^o$ and $(E,h)$ a holomorphic vector bundle on $X^o$ with a possibly singular hermitian metric. Suppose that, for every point $x\in X\backslash X^o$, there is a neighborhood $U_x$ of $x$ and a hermitian metric $ds^2_0$ on $U_x$ such that $ds^2_0|_{X^o\cap U_x}\lesssim ds^2|_{X^o\cap U_x}$. Then $\sD^{p,q}_{X,ds^2}(E,h)$ is a fine sheaf for each $p$ and $q$.
\end{lem}
\begin{proof}
	It suffices to show that for every $W\subset\overline{W}\subset U\subset X$ where $W$ and $U$ are small open subsets, there is a positive continuous function $f$ on $U$ such that
	\begin{itemize}
		\item ${\rm supp}(f)\subset \overline{W}$,
		\item $f$ is $C^\infty$ on $U\cap X^o$,
		\item $\dbar f$ has bounded fiberwise norm, with respect to the metric $ds^2$.
	\end{itemize}
	Choose a closed embedding $U\subset M$ where $M$ is a smooth complex manifold. Let $V\subset\overline{V}\subset M$ where $V$ is an open subset such that $V\cap U=W$. Let $ds^2_M$ be a hermitian metric on $M$ so that $ds^2_0|_{U\cap X^o}\sim ds^2_M|_{U\cap X^o}$. Let $g$ be a positive smooth function on $M$ whose support lies in $\overline{V}$. Denote $f=g|_U$. Then ${\rm supp}(f)\subset \overline{W}$ and $f$ is $C^\infty$ on $U\cap X^o$. It suffices to show the boundedness of the fiberwise norm of $\dbar f$. Since $U\cap X^o\subset M$ is a submanifold, one has the orthogonal decomposition 
	\begin{align*}
	T_{M,x}=T_{U\cap X^o,x}\oplus T_{U\cap X^o,x}^\bot, \quad \forall x\in U\cap X^o.
	\end{align*} 
	Therefore $|\dbar f|_{ds^2}\lesssim|\dbar f|_{ds^2_0}\leq |\dbar g|_{ds^2_M}<\infty$. The lemma is proved.
\end{proof}
The following estimate, which is essentially due to H\"ormander in \cite{Hormander1965} and Andreotti-Vesentini in \cite{AV1965} and developed by Demailly in \cite{Demailly1982}, plays a crucial role in proving various types of vanishing theorems in the present paper.
\begin{thm}\cite[Theorem 5.1]{Demailly1982}\label{thm_Hormander_incomplete}
	Let $Y$ be a complex manifold of dimension $n$ which admits a complete K\"ahler metric. Let $(E,h_\varphi)$ be a hermitian vector bundle with a possibly singular hermitian metric $h_\varphi:=e^{-\varphi}h_0$. Assume that $$\sqrt{-1}\Theta_{h_\varphi}(E):=\sqrt{-1}\ddbar\varphi\otimes {\rm Id}_{E}+\sqrt{-1}\Theta_{h_0}(E)\geq \omega\otimes {\rm Id}_{E}$$ for some (not necessarily complete) K\"ahler form $\omega$ on $Y$. Then for every $q>0$ and every $\alpha\in L^{n,q}_{(2)}(Y,E;\omega,h_\varphi)$ such that $\dbar\alpha=0$, there is $\beta\in L^{n,q-1}_{(2)}(Y,E;\omega,h_\varphi)$ such that $\dbar\beta=\alpha$ and $\|\beta\|^2_{\omega,h_\varphi}\leq q^{-1}\|\alpha\|^2_{\omega,h_\varphi}$.
\end{thm}
The above theorem works effectively locally on complex analytic singularities due to the following lemma by Grauert \cite{Grauert1956} (see also \cite[Lemma 2.4]{Pardon_Stern1991}).
\begin{lem}\label{lem_complete_metric_exists_locally}
	Let $x$ be a point of a complex space $X$ and let $X^o\subset X_{\rm reg}$ be a Zariski open subset. Then there is a neighborhood $U$ of $x$ and a complete K\"ahler metric on $U\cap X^o$. 
\end{lem}
\section{multiplier $S$-sheaf}
\subsection{Adjoint $L^2$ extension of a hermitian bundle}
Let $X$ be a complex space of dimension $n$ and $X^o\subset X_{\rm reg}$ a Zariski open subset. Let $ds^2$ be a hermitian metric on $X^o$.
\begin{defn}
	Let $(E,h)$ be a holomorphic vector bundle on $X^o$ with a possibly singular metric. Define
	\begin{align*}
	S_X(E,h):={\rm Ker}\left(\sD_{X,ds^2}^{n,0}(E,h)\stackrel{\dbar}{\to}\sD_{X,ds^2}^{n,1}(E,h)\right).
	\end{align*}
\end{defn}
The following proposition shows that $S_X(E,h)$ is independent of $ds^2$. Thus  $ds^2$ is omitted in the notation $S_X(E,h)$. Such property has already been discovered by Ohsawa in \cite{Ohsawa1980}. 
\begin{prop}\label{prop_L2_extension_independent}
	Let $(E,h)$ be a holomorphic vector bundle on $X^o$ with a possibly singular metric. Then $S_X(E,h)$ is independent of $ds^2$. 
\end{prop}
\begin{proof}
	Let $\pi:\tilde{X}\to X$ be a desingularization of $X$ so that $\widetilde{X}$ is smooth and $\pi$ is biholomorphic over $X^o$. We identify $X^o$ and $\pi^{-1}(X^o)$ for simplicity. Let $ds^2_{\tilde{X}}$ be a hermitian metric on $\tilde{X}$. Since $\pi$ is a proper map, a section of $K_{X^o}\otimes E$ is locally square integrable at $x$ if and only if it is locally square integrable near every point of $\pi^{-1}\{x\}$. It is therefore sufficient to show that 
	\begin{align}\label{align_S_independent}
	{\rm Ker}\left(\sD_{\tilde{X},\pi^\ast ds^2}^{n,0}(E,h)\stackrel{\dbar}{\to}\sD_{\tilde{X},\pi^\ast ds^2}^{n,1}(E,h)\right)={\rm Ker}\left(\sD_{\tilde{X},ds^2_{\tilde{X}}}^{n,0}(E,h)\stackrel{\dbar}{\to}\sD_{\tilde{X},ds^2_{\tilde{X}}}^{n,1}(E,h)\right).
	\end{align}	
	Since the problem is local, we assume that there is an orthogonal frame of cotangent fields $\delta_1,\dots,\delta_n$ such that 
	\begin{align}\label{align_lem_metric1}
	\pi^\ast ds^2\sim\lambda_1\delta_1\overline{\delta_1}+\cdots+\lambda_n\delta_n\overline{\delta_n}
	\end{align}
	and
	\begin{align}\label{align_lem_metric2}
	ds^2_{\tilde{X}}\sim\delta_1\overline{\delta_1}+\cdots+\delta_n\overline{\delta_n}.
	\end{align}
	Let $s=\delta_1\wedge\cdots\wedge\delta_n\otimes\xi$.
	It follows from (\ref{align_lem_metric1}) and (\ref{align_lem_metric2})  that
	\begin{align}\label{align_ind_ds2}
	\|s\|^2_{\pi^\ast ds^2,h}&=\int|\delta_1\wedge\cdots\wedge\delta_n\otimes\xi|^2_{\pi^\ast ds^2,h}\prod_{i=1}^n\lambda_i\delta_i\wedge\overline{\delta_i}\\\nonumber
	&=\int|\xi|^2_{h}\prod_{i=1}^n\delta_i\wedge\overline{\delta_i}\\\nonumber
	&=\|s\|^2_{ds^2_{\tilde{X}},h}.
	\end{align}
	Therefore $\|s\|^2_{\pi^\ast ds^2,h}$ is locally finite if and only if $\|s\|^2_{ds^2_{\tilde{X}},h}$ is locally finite. This proves (\ref{align_S_independent}).
\end{proof}
Similar to Saito's $S$-sheaf \cite{MSaito1991} and the multipler ideal sheaf \cite[Proposition 5.8]{Demailly2012}, $S_X(E,h)$ has the functoriality property.
\begin{prop}[Functoriality]\label{prop_L2ext_birational}
	Let $(E,h)$ be a holomorphic vector bundle on $X^o$ with a possibly singular hermitian metric. Let $\pi:X'\to X$ be a proper holomorphic map between complex spaces which is biholomorphic over $X^o$. Then $$\pi_\ast S_{X'}(\pi^\ast E,\pi^\ast h)=S_X(E,h).$$
\end{prop}
\begin{proof}
	By Proposition \ref{prop_L2_extension_independent}, 
	$$S_{X'}(\pi^\ast E,\pi^\ast h)={\rm Ker}\left(\sD_{X',\pi^\ast ds^2}^{n,0}(\pi^\ast E,\pi^\ast h)\stackrel{\dbar}{\to}\sD_{X',\pi^\ast ds^2}^{n,1}(\pi^\ast E,\pi^\ast h)\right)$$
	and
	$$S_{X}(E,h)={\rm Ker}\left(\sD_{X, ds^2}^{n,0}(E,h)\stackrel{\dbar}{\to}\sD_{X, ds^2}^{n,1}(E,h)\right).$$
	Since $\pi$ is a proper map, a section of $E$ is locally square integrable at $x$ if and only if it is locally square integrable near each point of $\pi^{-1}\{x\}$. The lemma is proved.
\end{proof}
\subsection{Multiplier $S$-sheaf}
Throughout this subsection we assume that $X$ is a complex space, $X^o\subset X_{\rm reg}$ is a Zariski open subset and $\bV=(\cV,\nabla,\cF^\bullet,h_{\bV})$ is an $\bR$-polarized variation of Hodge structure on $X^o$. Let $\varphi$ be a quasi-psh function on $X$. 
\begin{defn}\label{defn_Svarphi}
	The multiplier $S$-sheaf associated to $\bV$ and $\varphi$ is defined by
	$$S(IC_X(\bV),\varphi):=S_X(S(\bV),e^{-\varphi}h_\bV).$$
\end{defn}
\begin{lem}\label{lem_S_change_varphi_bounded}
	If $\varphi_1-\varphi_2$ is locally bounded over $X$, then  $S(IC_X(\bV),\varphi_1)=S(IC_X(\bV),\varphi_2)$.
\end{lem}
\begin{proof}
	By assumption, we know that 
	$$\int|\alpha|^2e^{-\varphi_1}{\rm vol}\sim \int|\alpha|^2e^{-\varphi_2}{\rm vol}$$
	for every local section $\alpha$ of $K_{X^o}\otimes S(\bV)$. Thus we prove the lemma.
\end{proof}
The proof following lemma is straightforward. Here we omit its proof.
\begin{lem}\label{lem_integral}
Let $f$ be a holomorphic function on $\Delta^\ast_{\frac{1}{2}}:=\{z\in\bC||z|<\frac{1}{2}\}$ and $a\in\bR$. Then
$$\int_{\Delta^\ast_{\frac{1}{2}}}|f|^2|z|^{2a}dzd\bar{z}<\infty$$
if and only if $v(f)+a>-1$. Here
$$v(f):=\min\{l|f_l\neq0\textrm{ in the Laurent expansion } f=\sum_{i\in\bZ}f_iz^i\}.$$
\end{lem}
\begin{defn}\label{defn_generalized_analytic_singularity}
	A quasi-psh function $\varphi$ on $X$ has generalized analytic singularities along a closed analytic subspace $Z\subset X$ if, for every point $x\in Z$, there is a neighborhood $U$ of $x$, some holomorphic functions $g_1,\dots,g_r\in\sO_X(U)$ and some real numbers $b\geq0$, $a_1,\dots,a_r\geq0$ such that
	\begin{align*}
	\exp(\varphi|_U)\sim\left(\sum_{i=1}^r|g_i|^{2a_i}\right)^b.
	\end{align*}
	$\varphi$ has analytic singularities if $a_i=1,\forall i=1,...,r$.
	When $\varphi$ has (generalized) analytic singularities along the entire $X$, we briefly say that it has (generalized) analytic singularities.
\end{defn}
The following proposition elucidates the relation between $S(IC_X(\bV),\varphi)$ and multiplier ideal sheaves when $\varphi$ has generalized analytic singularities. 
\begin{prop}\label{prop_key_est}
	Assume that $X=\Delta^n$ is the polydisc. Denote $E:=\{z_1\cdots z_r=0\}$, denote $E_i:=\{z_i=0\}, \forall i=1,\dots,r$ and denote $j:X^o:=X\backslash E\to X$ to be the inclusion.  Let $\bV=(\cV,\nabla,\cF^\bullet,h_\bV)$ be an $\bR$-polarized variation of Hodge structure on $X^o$ and let $\varphi$ be a quasi-psh function on $X$ which has generalized analytic singularities. Let ${\bf 0}=(0,\dots,0)\in X$ and let $\widetilde{v_1},\dots,\widetilde{v_m}$ be an $L^2$-adapted frame of $R(IC_X(\bV))$ locally at ${\bf 0}$ as in Proposition \ref{prop_adapted_frame}. Let $f_1,\dots,f_m\in (j_\ast\sO_{X^o})_{\bf 0}$. Then 
	$$\sum_{i=1}^mf_i[\widetilde{v_i}dz_1\wedge\cdots\wedge dz_n]_{\bf 0}\in S(IC_X(\bV),\varphi)_{\bf 0}$$ if and only if
	$$f_i\in\sI\big(\varphi-\sum_{j=1}^r2\alpha_{E_j}(\widetilde{v_i})\log|z_j|\big)_{\bf 0}$$
	for every $i=1,\dots,m$. In conclusion, there is an isomorphism
	\begin{align*}
	S(IC_{X}(\bV),\varphi)_{\bf 0}\simeq \omega_{X,{\bf 0}}\otimes\bigoplus_{i=1}^m\sI(\varphi-\sum_{j=1}^r2\alpha_{E_j}(\widetilde{v_i})\log|z_j|)_{\bf 0}\widetilde{v_i}.
	\end{align*}
\end{prop}
\begin{proof}
	Denote $ds^2_0=\sum_{i=1}^ndz_id\bar{z}_i$.
	Since $\widetilde{v_1},\dots,\widetilde{v_m}$ is an $L^2$-adapted frame as in Proposition \ref{prop_adapted_frame}, the integral
	$$\int|\sum_{i=1}^mf_i\widetilde{v_i}dz_1\wedge\cdots\wedge dz_n|^2e^{-\varphi}{\rm vol}_{ds^2_0}=\int|\sum_{i=1}^mf_i\widetilde{v_i}|^2e^{-\varphi}{\rm vol}_{ds^2_0}$$ is finite near ${\bf 0}$ if and only if
	$$\int|f_i\widetilde{v_i}|^2e^{-\varphi}{\rm }{\rm vol}_{ds^2_0}\sim \int|f_i|^2e^{-\varphi}{\rm }\prod_{j=1}^r|z_j|^{2\alpha_{E_j}(\widetilde{v_i})}\lambda_i{\rm vol}_{ds^2_0}$$
	is finite near ${\bf 0}$ for every $i=1,\dots,m$. Here $\lambda_i$ is a positive real function so that
	\begin{align}\label{align_keylem_lambda}
	1\lesssim \lambda_i\lesssim |z_1\cdots z_r|^{-\epsilon},\quad\forall\epsilon>0.
	\end{align}	
	We are going to show that $|f_i\widetilde{v_i}|^2e^{-\varphi}{\rm }{\rm vol}_{ds^2_0}$ is locally integrable if and only if $f_i\in\sI\big(\varphi-\sum_{j=1}^r2\alpha_{E_j}(\widetilde{v_i})\log|z_i|\big)$.

	Since $\varphi$ has generalized analytic singularities, $Z:=\{e^\varphi=0\}$ is a closed analytic subspace. Let $\pi:\widetilde{X}\to X$ be a desingularization so that
	\begin{enumerate}
		\item $\pi$ is biholomorphic over $X^o\backslash Z'$ where $Z'$ is the union of  the  irreducible components of $Z$ that is not a component of $E$.
		\item $\pi^{-1}(Z\cup E)$ and $\pi^{-1}(E)$ are normal crossing divisors on $\widetilde{X}$.
	\end{enumerate}
	Let $w_1,\dots,w_n$ be  holomorphic local coordinates of $\widetilde{X}$ such that $\pi^{-1}(Z\cup E)=\{w_1\cdots w_s=0\}$. Then we obtain that 
	\begin{align*}
	e^{\pi^\ast\varphi}\sim \prod_{i=1}^s|w_i|^{2a_i},\quad \pi^\ast\left(\prod_{j=1}^r|z_j|^{2\alpha_{E_j}(\widetilde{v_i})}\right)\sim \prod_{i=1}^s|w_i|^{-2b_i}
	\end{align*}
	for some nonnegative constants $a_1,\dots,a_s,b_1,\dots,b_s$. Let $ds^2_1$ be a hermitian metric on $\widetilde{X}$. Then 
	$$\pi^\ast{\rm vol}_{ds^2_0}\sim \prod_{i=1}^s|w_i|^{2c_i}{\rm vol}_{ds^2_1}$$
	for some nonnegative constants $c_1,\dots,c_s$.
	If $f_i\in\sI\big(\varphi-\sum_{j=1}^r2\alpha_{E_j}(\widetilde{v_i})\log|z_j|\big)$, then
	\begin{align}\label{align_keylem_1}
	\int\pi^\ast\left(| f_i|^2e^{-\varphi+\sum_{j=1}^r2\alpha_{E_j}(\widetilde{v_i})\log|z_j|}{\rm vol}_{ds^2_0}\right)\sim \int|\pi^\ast f_i|^2\prod_{i=1}^s|w_i|^{2(-a_i-b_i+c_i)}{\rm vol}_{ds^2_1}
	\end{align}
	is locally integrable. 
	Denote 
	$$v_j(g):=\min\{l|g_l\neq0\textrm{ in the Laurent expansion } g=\sum_{i\in\bZ}g_iw_j^i\}, \quad j=1,\dots,s.$$
	By Lemma \ref{lem_integral}, the local integrability of (\ref{align_keylem_1}) implies that
	\begin{align*}
	v_j(\pi^\ast f_i)-a_j-b_j+c_j>-1+\epsilon,\quad j=1,\dots, s
	\end{align*}
	for some $\epsilon>0$.
	By (\ref{align_keylem_lambda}), we obtain that
	\begin{align*}
	\int\pi^\ast\left(|f_i\widetilde{v_i}|^2e^{-\varphi}{\rm }{\rm vol}_{ds^2_0}\right)&\sim \int|\pi^\ast f_i|^2\prod_{j=1}^s|w_j|^{2(-a_j-b_j+c_j)}\pi^\ast\lambda_i{\rm vol}_{ds^2_1}\\\nonumber
	&\lesssim \int|\pi^\ast f_i|^2\prod_{j=1}^s|w_j|^{2(-a_j-b_j+c_j-\epsilon)}{\rm vol}_{ds^2_1}
	\end{align*}
	is locally integrable. Since $\pi$ is a proper map, we see that $\int|f_i\widetilde{v_i}|^2e^{-\varphi}{\rm }{\rm vol}_{ds^2_0}$ is locally integrable.
	
	Conversely, let $f_i$ be a holomorphic function on $X^o$ such that $\int|f_i\widetilde{v_i}|^2e^{-\varphi}{\rm }{\rm vol}_{ds^2_0}$ is locally integrable. We know that
	\begin{align*}
	\int| f_i|^2e^{-\varphi+\sum_{j=1}^r2\alpha_{E_j}(\widetilde{v_i})\log|z_j|}{\rm vol}_{ds^2_0}\lesssim \int|f_i|^2e^{-\varphi}{\rm }\prod_{j=1}^r|z_j|^{2\alpha_{E_j}(\widetilde{v_i})}\lambda_i{\rm vol}_{ds^2_0}\sim \int|f_i\widetilde{v_i}|^2e^{-\varphi}{\rm }{\rm vol}_{ds^2_0}
	\end{align*}
	is locally integrable by (\ref{align_keylem_lambda}). The proof is finished.
\end{proof}
\begin{prop}\label{prop_S_vs_mulidealsheaf}
Assume that $X$ is a complex manifold and $X\backslash X^o\subset X$ is a (possibly empty) normal crossing divisor. Let $\bV$ be an $\bR$-polarized variation of Hodge structure on $X^o$ which has unipotent local monodromies. Let
$\varphi$ be a psh function with generalized analytic singularities along $X\backslash X^o$. Then there is an isomorphism $$S(IC_X(\bV),\varphi)\simeq S(IC_X(\bV))\otimes \sI(\varphi).$$
\end{prop}
\begin{proof}
		Since $S(IC_X(\bV))$ is locally free, it suffices to show that
	\begin{align}\label{align_prop_S_multiideal_key}
	S(IC_X(\bV),\varphi)=\sI(\varphi)S(IC_X(\bV))
	\end{align}
	as subsheaves of $j_\ast(\omega_{X^o}\otimes S(\bV))$. Here we regard $S(IC_X(\bV)):= (j_\ast( S(\bV))\cap \cV_{-1})\otimes\omega_{X^o}$ as a subsheaf of $j_\ast(\omega_{X^o}\otimes S(\bV))$.
	
	Since the problem is local, we assume that  $X\subset\bC^n$ is a germ of open subset at ${\bf 0}=(0,\dots,0)$ with the standard coordinates $z_1,...,z_n$ such that $E:=X\backslash X^o=\{z_1\cdots z_r=0\}$. Denote $E_i:=\{z_i=0\}, \forall i=1,\dots,r$. Let $\widetilde{v_1},\dots,\widetilde{v_m}$ be an $L^2$-adapted frame of $R(IC_X(\bV))$ locally at ${\bf 0}$ as in Proposition \ref{prop_adapted_frame} and let $f_1,\dots,f_m\in (j_\ast\sO_{X^o})_{\bf 0}$.  Proposition \ref{prop_key_est} tells us  that
	$$\sum_{i=1}^mf_i[\widetilde{v_i}dz_1\wedge\cdots\wedge dz_n]_x\in S(IC_X(\bV),\varphi)_{\bf 0}$$ if and only if
	$$f_i\in\sI\big(\varphi-\sum_{j=1}^r2\alpha_{E_j}(\widetilde{v_i})\log|z_j|\big)_{\bf 0},\quad i=1,\dots,m.$$
	Since the local monodromies of $\bV$ are unipotent, $\alpha_{E_j}(\widetilde{v_i})=0$, $\forall j=1,\dots,r$. Hence we prove (\ref{align_prop_S_multiideal_key}).
\end{proof}
\begin{thm}\label{cor_S_phi=0}
	$S(IC_X(\bV),0)\simeq S(IC_X(\bV))$. In particular, $S(IC_X(\bV))$ is independent of the choice of the desingularization.
\end{thm}
\begin{proof}
	Let $\widetilde{X}\to X$ be a desingularization so that $\pi$ is biholomorphic over $X^o$ and $\pi^{-1}(X\backslash X^o)$ is a simple normal crossing divisor. We claim that
	\begin{align}\label{align_thm_SS_key}
	S(IC_{\widetilde{X}}(\pi^\ast\bV),0)\simeq  S(IC_{\widetilde{X}}(\pi^\ast\bV)).
	\end{align}
	If the claim is true, it follows from Proposition \ref{prop_L2ext_birational}   that
	$$S(IC_X(\bV),0)\simeq\pi_\ast S(IC_{\widetilde{X}}(\pi^\ast\bV),0)\simeq \pi_\ast S(IC_{\widetilde{X}}(\pi^\ast\bV))\simeq S(IC_X(\bV)).$$
	
	Now it suffices to show (\ref{align_thm_SS_key}). Since the problem is local, we assume that $\widetilde{X}=\Delta^n$ is the polydisc with the standard holomorphic coordinates $z_1,\dots,z_n$ such that $E:=\pi^{-1}(X\backslash X^o)=\{z_1\cdots z_r=0\}$. Let  $j:\widetilde{X}\backslash E\to \widetilde{X}$ denote the inclusion and denote $E_i:=\{z_i=0\}, \forall i=1,\dots,r$. Let ${\bf 0}=(0,\dots,0)\in \widetilde{X}$ and let  $\widetilde{v_1},\dots,\widetilde{v_m}$ be an $L^2$-adapted frame of $R(IC_{\widetilde{X}}(\pi^\ast\bV))$ locally at ${\bf 0}$ as in Proposition \ref{prop_adapted_frame}. Let $f_1,\dots,f_m\in (j_\ast\sO_{\pi^{-1}(X^o)})_{\bf 0}$. Proposition \ref{prop_key_est} shows  that 
	$$\sum_{i=1}^mf_i[\widetilde{v_i}dz_1\wedge\cdots\wedge dz_n]_{\bf 0}\in S(IC_X(\bV),0)_{\bf 0}$$ if and only if
	$$f_i\in\sI\big(-\sum_{j=1}^r2\alpha_{E_j}(\widetilde{v_i})\log|z_j|\big)_{\bf 0}$$
	for every $i=1,\dots,m$. Since $\alpha_{E_j}(\widetilde{v_i})\in(-1,0], \forall j=1,\dots,r$, Lemma \ref{lem_integral} shows that
	${\rm exp}\left(\sum_{j=1}^r2\alpha_{E_j}(\widetilde{v_i})\log|z_j|\right)$ is integrable. Hence
	\begin{align*}
	\sI\big(-\sum_{j=1}^r2\alpha_{E_j}(\widetilde{v_i})\log|z_j|\big)_{\bf 0}\simeq\sO_{\widetilde{X},\bf 0},\quad i=1,\dots,m.
	\end{align*}
The proof of the theorem is finished.
\end{proof}
\begin{prop}\label{prop_S_coherent}
	$S(IC_X(\bV),\varphi)$ is a coherent subsheaf of $S(IC_X(\bV))$ for an arbitrary general quasi-psh function $\varphi$.
\end{prop}
\begin{proof}
    First we show that $S(IC_X(\bV),\varphi)\subset S(IC_X(\bV))$. By Theorem \ref{cor_S_phi=0}, it is equivalent to show that 
    \begin{align}\label{align_S_phi_subset_S}
    S(IC_X(\bV),\varphi)\subset S(IC_X(\bV),0).
    \end{align}
    Since $\varphi$ is quasi-psh, it has an upper bound $ c_x$ locally at every point $x\in X$. Thus we have
    $$\int|\alpha|^2{\rm vol}\sim\int|\alpha|^2e^{-c_x}{\rm vol}\leq\int|\alpha|^2e^{-\varphi}{\rm vol}$$
    locally at $x$.
    This proves (\ref{align_S_phi_subset_S}).
    
    Next we prove its coherence. Let $\pi:\widetilde{X}\to X$ be a desingularization so that $\pi$ is biholomorphic over $X^o$ and $E:=\pi^{-1}(X\backslash X^o)$ is a simple normal crossing divisor. By abuse of notations we regard $X^o\subset \widetilde{X}$ as a subset. Denote $\psi=\pi^\ast\varphi$. Then $\psi$ is a quasi-psh function. There is an isomorphism
    \begin{align*}
    S(IC_X(\bV),\varphi)\simeq \pi_\ast\left(S(IC_{\widetilde{X}}(\bV),\psi)\right)
    \end{align*}
     by Proposition \ref{prop_L2ext_birational}.
    Since $\pi$ is proper, it suffices to show that $S(IC_{\widetilde{X}}(\bV),\psi)$ is a coherent sheaf on $\widetilde{X}$. Since the problem is local and $\widetilde{X}$ is smooth, we may assume that $\widetilde{X}\subset\bC^n$ is the unit ball so that $E=\{z_1\cdots z_s=0\}$ and $\psi$ has an upper bound. Denote $E_i=\{z_i=0\}, \forall i=1,\dots,s$.
   Notice that there is a complete K\"ahler metric on $X^o$  by Lemma \ref{lem_complete_metric_exists_locally}. Since $S(IC_{\widetilde{X}}(\bV))$ is coherent,  the space $\Gamma(\widetilde{X},S(IC_{\widetilde{X}}(\bV),\psi))$ generates a coherent subsheaf $\sJ$ of $S(IC_{\widetilde{X}}(\bV))$ by the strong Noetherian property. By the construction we have the inclusion 
    $\sJ\subset S(IC_{\widetilde{X}}(\bV),\psi)$. It remains to prove the converse. By Krull's theorem (\cite[Corollary 10.19]{Atiyah1969}), it suffices to show that
    \begin{align}\label{align_Krull}
    \sJ_x+S(IC_{\widetilde{X}}(\bV),\psi)_x\cap m_{\widetilde{X},x}^{k+1}S(IC_{\widetilde{X}}(\bV))=S(IC_{\widetilde{X}}(\bV),\psi)_x,\quad\forall k\geq0.
    \end{align}
    Let $\alpha\in S(IC_{\widetilde{X}}(\bV),\psi)_x$ be defined in a precompact neighborhood $V$ of $x$. Choose a $C^\infty$ cut-off function $\lambda$ so that $\lambda\equiv1$ near $x$ and ${\rm supp}\lambda\subset V$. Denote $|z|^2:=\sum_{i=1}^n|z_i|^2$. Since $\psi$ is quasi-psh, there is a constant $c>0$ such that $\psi(z)+c|z|^2$ is psh.
    Let
    \begin{align*}
    \psi_k(z):=\psi(z)+2(n+k+1)\log|z-x|+(c+1)|z|^2
    \end{align*}
    and $h_{\psi_k}=e^{-\psi_k}h_\bV$. Denote $\omega_0:=\sqrt{-1}\ddbar|z|^2$.
    By Theorem \ref{thm_geq0_S(V)}, we have
    \begin{align*}
    \sqrt{-1}\Theta_{h_{\psi_k}}(S(\bV))= \sqrt{-1}\ddbar\psi_k+\sqrt{-1}\Theta_{h_{\bV}}(S(\bV))\geq \omega_0.
    \end{align*}
    Since ${\rm supp}\lambda\alpha\subset V$ and $\dbar(\lambda\alpha)=0$ near $x$, we have
    \begin{align*}
    \|\dbar(\lambda\alpha)\|^2_{\omega_0,h_{\psi_k}}\sim \|\dbar(\lambda\alpha)\|^2_{\omega_0,h_{\psi}}\leq\|\dbar\lambda\|^2_{L^\infty}\|\alpha\|^2_{\omega_0,h_{\psi}}+|\lambda|^2\|\dbar\alpha\|^2_{\omega_0,h_{\psi}}<\infty
    \end{align*}
    Hence, Theorem \ref{thm_Hormander_incomplete} provides a solution of the equation $\dbar\beta=\dbar(\lambda\alpha)$ so that
    \begin{align}\label{align_norm_beta_psi}
    \|\beta\|^2_{\omega_0,h_\psi}\lesssim\int_{X^o}|\beta|^2_{\omega_0,h_\bV}e^{-\psi}|z-x|^{-2(n+k+1)}{\rm vol}_{\omega_0}\lesssim \|\dbar(\lambda\alpha)\|^2_{\omega_0,h_{\psi_k}}<\infty.
    \end{align}
    Thus $\gamma=\lambda\alpha-\beta$ is holomorphic and $\gamma\in \Gamma(\widetilde{X},S(IC_{\widetilde{X}}(\bV),\psi))$. 
    Using the notations in \S \ref{section_Hodge_module}, we have 
    \begin{align*}
    S(IC_{\widetilde{X}}(\bV))=\left(j_\ast(S(\bV))\cap\widetilde{\cV}_{-1}\right)\otimes\omega_{\widetilde{X}}
    \end{align*}
    where $j:X^o\to \widetilde{X}$ is the open immersion.
    Since $\psi$ has an upper bound, we have
    $$\|\beta\|^2_{\omega_0,h_\bV}\lesssim\int_{X^o}|\beta|^2_{\omega_0,h_\bV}|z-x|^{-2(n+k+1)}{\rm vol}_{\omega_0}\lesssim \int_{X^o}|\beta|^2_{\omega_0,h_\bV}e^{-\psi}|z-x|^{-2(n+k+1)}{\rm vol}_{\omega_0}<\infty.$$
    Hence $\beta\in S(IC_{\widetilde{X}}(\bV))\subset \widetilde{\cV}_{-1}$ by Theorem \ref{cor_S_phi=0}. 
    
    Let $\widetilde{{\bf v}}=(\widetilde{v}_1,\dots, \widetilde{v}_r)$ be the local frame of $R(IC_{\widetilde{X}}(\bV))\subset\widetilde{\cV}_{-1}$ where ${\bf v}=e^{-\sum_{i}2\pi\sqrt{-1}z_i{\rm Res}_{E_i}\widetilde{\nabla}}\widetilde{{\bf v}}$ is an orthogonal basis of $\bV$ (c.f. \S \ref{section_L2_adapted_frame}). Then there are holormorphic functions $f_1,\dots,f_r\in\sO_{\widetilde{X}}(X^o)$ such that $$\beta=\sum_{i=1}^r f_i\widetilde{v}_idz_1\wedge\cdots\wedge dz_n.$$
    It follows from Theorem \ref{thm_Hodge_metric_asymptotic} that 
    \begin{align*}
    (z_1\cdots z_s)^{\epsilon}\lesssim|v_i| ,\quad\forall\epsilon>0,\forall i=1,\dots,r.
    \end{align*}
    Thus one gets that 
    \begin{align*}
    \int_{X^o}\sum_{i=1}^r|f_i|^2(z_1\cdots z_r)^{2\epsilon}|z-x|^{-2(n+k+1)}{\rm vol}_{\omega_0}\lesssim\int_{X^o}|\beta|^2_{\omega_0,h_\bV}|z-x|^{-2(n+k+1)}{\rm vol}_{\omega_0}<\infty
    \end{align*}
     for every $\epsilon>0$.
    This implies that $f_i\in m_{\widetilde{X},x}^{k+1}$ for every $i=1,\dots,r$ (\cite[Lemma 5.6]{Demailly2012}). Hence
    $$\alpha_x-\gamma_x=\beta_x\in m_{\widetilde{X},x}^{k+1}S(IC_{\widetilde{X}}(\bV)).$$ 
    On the other hand, $\beta\in S(IC_{\widetilde{X}}(\bV),\psi)$ because of (\ref{align_norm_beta_psi}).
    This proves (\ref{align_Krull}).
\end{proof}
\begin{lem}\label{lem_Kernel}
	Let $(E,h_\varphi)$ be a holomorphic vector bundle  on $X$ with a possibly singular hermitian metric $h_\varphi=e^{-\varphi}h_0$. Then
	$S(IC_X(\bV),\varphi)\otimes E\simeq S_X(S(\bV)\otimes E,h_\bV\otimes h_{\varphi})$.
\end{lem}
\begin{proof}
	Let $x\in X$ be a point and let $U$ be an open neighborhood of $x$ so that $E|_U\simeq \sO_U^{\oplus r}$ and $h_0$ is quasi-isometric to the trivial metric, i.e.
	$$|\sum_{i=1}^ra_ie_i|^2_{h_0}\sim \sum_{i=1}^r|a_i|^2$$
	where $\{e_i\}$ is the standard frame of $\sO_U^{\oplus r}$ and $a_i$s are measurable functions on $U$.
	Let $ds^2$ be an arbitrary hermitian metric on $X^o$ and let $\alpha=\sum_{i=1}^r\alpha_i\otimes e_i$ be a measurable section of $S(\bV)\otimes E|_{U\cap X^o}$. Then
	$$\|\alpha\|^2_{ds^2,h_\bV\otimes h_\varphi}\sim\sum_{i=1}^r\|\alpha_i\|^2_{ds^2,e^{-\varphi}h_\bV}$$
	is finite if and only if each $\|\alpha_i\|^2_{ds^2,e^{-\varphi}h_\bV}$ is finite. This proves the lemma.
\end{proof}
We end this section by proving an approximation property of $S(IC_X(\bV),\varphi)$. 
\begin{prop}\label{prop_approx_S}
	Let $\varphi$ and $\psi$ be quasi-psh functions on $X$ which have generalized analytic singularities. Then
	$$S(IC_X(\bV),\varphi)=S(IC_X(\bV),\varphi+\epsilon\psi),\quad 0<\epsilon\ll1.$$
\end{prop}
\begin{proof}
	Since the problem is local, we assume that $X$ is a germ of complex space.
	Since $\varphi$ and $\psi$ have generalized analytic singularities, $Z_1:=\{e^\varphi=0\}$ and $Z_2:=\{e^\psi=0\}$ are closed analytic subspaces. Let $\pi:\widetilde{X}\to X$ be a desingularization so that
	\begin{enumerate}
		\item $\pi$ is biholomorphic over $X^o\backslash (Z_1\cup Z_2)$;
		\item $\pi^{-1}(Z_1\cup Z_2\cup(X\backslash X^o))$ and $\pi^{-1}(X\backslash X^o)$ are normal crossing divisors of $\widetilde{X}$.
	\end{enumerate}   
    Let $w_1,\dots,w_n$ be  holomorphic local coordinates on $\widetilde{X}$ such that $\pi^{-1}(Z_1\cup Z_2\cup(X\backslash X^o))=\{w_1\cdots w_s=0\}$ and $\pi^{-1}(X\backslash X^o)=\{w_1\cdots w_r=0\}$ with $0\leq r\leq s$. Notice that $r=0$ if  we consider the problem on $X^o$. Denote $E_i=\{z_i=0\}, \forall i=1,\dots,r$. Then we know that 
    \begin{align*}
    e^{\pi^\ast\varphi}\sim \prod_{i=1}^s|w_i|^{2a_i}, \quad e^{\pi^\ast\psi}\sim \prod_{i=1}^s|w_i|^{2b_i}
    \end{align*}
    for some nonnegative constants $a_1,\dots,a_s,b_1,\dots,b_s$.
    Denote $j:\pi^{-1}(X^o)\to\widetilde{X}$ to be the open immersion. Let $\widetilde{v_1},\dots,\widetilde{v_m}$ be an $L^2$-adapted frame of $R(IC_{\widetilde{X}}(\pi^\ast\bV))$ locally at ${\bf 0}=(0,\dots,0)\in \widetilde{X}$ as in Proposition \ref{prop_adapted_frame}. Let $f_1,\dots,f_m\in (j_\ast\sO_{\pi^{-1}(X^o)})_{\bf 0}$ and $\alpha=\sum_{i=1}^mf_i\widetilde{v_i}dz_1\wedge\cdots\wedge dz_n$. Then Proposition \ref{prop_key_est} shows that
    $$[\alpha]_x\in S(IC_{\widetilde{X}}(\pi^\ast\bV),\pi^\ast\varphi)_{\bf 0}\quad \left(\textrm{resp. }[\alpha]_x\in S(IC_{\widetilde{X}}(\pi^\ast\bV),\pi^\ast\varphi+\pi^\ast\psi)_{\bf 0}\right)$$ if and only if
    $$f_i\in\sI\big(\pi^\ast\varphi-\sum_{j=1}^r2\alpha_{E_j}(\widetilde{v_i})\log|w_j|\big)_{\bf 0}\quad\bigg(\textrm{resp. }f_i\in\sI\big(\pi^\ast\varphi+\pi^\ast\psi-\sum_{j=1}^r2\alpha_{E_j}(\widetilde{v_i})\log|w_j|\big)_{\bf 0}\bigg)$$
    for every $i=1,\dots,m$. Hence $[\alpha]_x\in S(IC_{\widetilde{X}}(\pi^\ast\bV),\pi^\ast\varphi)_{\bf 0}$ is equivalent to that the integral
    \begin{align*}
    \int|f_i|^2\prod_{j=1}^s|w_j|^{-2a_j}\prod_{j=1}^r|w_j|^{2\alpha_{E_j}(\widetilde{v_i})}{\rm vol}_{ds^2_0}
    \end{align*}
    is finite near ${\bf 0}$ for every $i=1,\dots,m$. Denote
    $$v_j(f):=\min\{l|f_l\neq0\textrm{ in the Laurent expansion } f=\sum_{i\in\bZ}f_iw_j^i\}.$$
    By Lemma \ref{lem_integral}, we observe that $$f\in\sI\big(\pi^\ast\varphi-\sum_{j=1}^r2\alpha_{E_j}(\widetilde{v_i})\log|w_j|\big)$$ if and only if
    \begin{align}\label{align_est_varphi}
    v_j(f)-a_j+\alpha_{E_j}(\widetilde{v_i})>-1,\quad \forall j=1,\dots,s.
    \end{align}
    Here we set $\alpha_{E_j}(\widetilde{v_i})=0, \forall j=r+1,\dots,s$.
    
   Similar arguments show that $$f\in\sI\big(\pi^\ast(\varphi+\epsilon\psi)-\sum_{j=1}^r2\alpha_{E_j}(\widetilde{v_i})\log|w_j|\big)$$ if and only if
    \begin{align}\label{align_est_varphipsi}
    v_j(f)-a_j-\epsilon b_j+\alpha_{E_j}(\widetilde{v_i})>-1,\quad \forall j=1,\dots,s.
    \end{align}
    Conditions (\ref{align_est_varphi}) and (\ref{align_est_varphipsi}) are equivalent when $\epsilon>0$ is small enough. We obtain therefore that 
    $$S(IC_{\widetilde{X}}(\pi^\ast\bV),\pi^\ast\varphi)\simeq S(IC_{\widetilde{X}}(\pi^\ast\bV),\pi^\ast(\varphi+\epsilon\psi)),\quad 0<\epsilon\ll1.$$
    By Proposition \ref{prop_L2ext_birational}, 
    $$S(IC_X(\bV),\varphi)\simeq \pi_\ast(S(IC_{\widetilde{X}}(\pi^\ast\bV),\pi^\ast\varphi))\simeq\pi_\ast(S(IC_{\widetilde{X}}(\pi^\ast\bV),\pi^\ast(\varphi+\epsilon\psi)))\simeq S(IC_X(\bV),\varphi+\epsilon\psi)$$
    when $\epsilon>0$ is small enough.
\end{proof}
\subsection{Extension and adjunction}
In this section we consider the extension and adjunction properties for $S(IC_X(\bV),\varphi)$.  These  results will be used in proving the Demailly-Kawamata-Viehweg vanishing theorem for Saito's $S$-sheaf $S(IC_X(\bV))$ (Corollary \ref{cor_KV_vanishing}).
\begin{thm}[Ohsawa-Takegoshi extension theorem for $S(IC_X(\bV),\varphi)$]\label{thm_S_OT_extension}
	Let $X$ be a complex space and $\Omega\subset X_{\rm reg}$ a Zariski open subset which is a Stein manifold. Let $\bV$ be an $\bR$-polarized variation of Hodge structure on some Zariski open subset $X^o\subset X_{\rm reg}$. Let $H\subset X$ be a Cartier divisor such that $\sO_X(H)\simeq\sO_X$ and $H_{\rm reg}\cap \Omega\cap X^o$ is dense in $H_{\rm reg}$.  Let $(E,h_\varphi)$ be a vector bundle on $X$ with a singular hermitian metric $h_\varphi=e^{-\varphi}h_0$. Assume that $\varphi$ is smooth over some Zariski open subset of $X$ and $\sqrt{-1}\Theta_{h_{\varphi}}(E)\geq 0$ as a current. Let $h_H$ be a smooth hermitian metric on $\sO_X(H)$ with zero curvature. Then there is a constant $C>0$ such that, for every $\alpha\in \Gamma(H,S(IC_H(\bV|_{H\cap X^o}),\varphi|_H)\otimes E|_H\otimes\sO_X(H)|_H)$ satisfying that $\|\alpha\|_{h_{\bV}h_\varphi h_H}<\infty$, there is $\beta\in \Gamma(X,S(IC_X(\bV),\varphi)\otimes E)$ which satisfies that $\beta|_H=\alpha$ and $\|\beta\|_{h_{\bV} h_\varphi}\leq C\|\alpha\|_{h_{\bV} h_\varphi h_H}$.
\end{thm}
\begin{proof}
	Assume that $\varphi$ is smooth over a Zariski open subset $U\subset X$. Let $H'$ be a Cartier divisor of $\Omega$ so that $Y:=\Omega\backslash H'\subset X^o\cap U$ and $Y$ is a Stein manifold. Let $$\alpha\in \Gamma(H,S(IC_H(\bV|_{H\cap X^o}),\varphi|_H)\otimes E|_H\otimes\sO_X(H)|_H)$$ such that $\|\alpha\|_{h_{\bV} h_\varphi h_H}<\infty$. By the Ohsawa-Takegoshi extension theorem \cite{OT1988} (see also \cite[Theorem 2.2]{Guan-Zhou2015}), there is $\beta\in \Gamma(Y,K_{Y}\otimes S(\bV|_Y)\otimes E|_Y)$ such that $\beta|_{Y\cap H}=\alpha|_{Y\cap H}$ and $\|\beta\|_{Y,h_{\bV} h_\varphi}\leq C\|\alpha|_{Y\cap H}\|_{h_{\bV} h_\varphi h_H}$ for some constant $C>0$. It follows from Lemma \ref{lem_Kernel} that
	$$\beta\in \Gamma(X,S(IC_X(\bV|_Y),\varphi)\otimes E)=\Gamma(X,S(IC_X(\bV),\varphi)\otimes E).$$ Since $\alpha$ and $\beta$ are both  holomorphic sections, we get $\beta|_H=\alpha$.
\end{proof}
An immediate consequence of the extension theorem is the following 
\begin{cor}[Restriction Formula]
	Let $X$ be a complex space and $X^o\subset X_{\rm reg}$ a Zariski open subset. Let $\bV$ be an $\bR$-polarized variation of Hodge structure on some Zariski open subset $X^o\subset X_{\rm reg}$. Let $\varphi$ be a quasi-psh function on $X$ which is smooth over some Zariski open subset of $X$. Let $H$ be a reduced Cartier divisor of $X$ such that $H_{\rm reg}\cap X^o$ is dense in $H_{\rm reg}$. Then
	\begin{align}\label{align_restriction_formula}
	S(IC_H(\bV|_{H\cap X^o}),\varphi|_H)\subset S(IC_X(\bV),\varphi)|_H\otimes\sO_X(H)|_H
	\end{align}
	as subsheaves of $S(IC_X(\bV),0)|_H\otimes \sO_X(H)|_H$.
\end{cor}
\begin{proof}
	By Theorem \ref{thm_S_OT_extension}, there is an inclusion 
	$$S(IC_H(\bV|_{H\cap X^o}),\varphi|_H)\subset S_X(S(\bV)\otimes \sO_X(H),h_\bV\otimes e^{-\varphi}h_0)|_H$$
	where $h_0$ is an arbitrary smooth hermitian metric on $\sO_X(H)$. By Lemma \ref{lem_Kernel} there is an isomorphism
	$$S_X(S(\bV)\otimes \sO_X(H),h_\bV\otimes e^{-\varphi}h_0)\simeq S(IC_X(\bV),\varphi)\otimes\sO_X(H).$$
    The corollary is proved.
\end{proof}
Generally, the inclusion (\ref{align_restriction_formula}) is strict even if $\varphi=0$. From the perspective of Hodge modules, $IC_X(\bV)|_H$ could be a mixed Hodge module while $IC_H(\bV|_{H\cap X^o})$ is pure. When $\varphi$ has generalized analytic singularities, (\ref{align_restriction_formula}) is an equality if $H$ is in a general position.
\begin{prop}\label{prop_adjunction}
	Let $X$ be a projective algebraic variety and let $\bV$ be an $\bR$-polarized variation of Hodge structure on some Zariski open subset $X^o\subset X_{\rm reg}$. Let $\varphi$ be a quasi-psh function on $X$ which has generalized analytic singularities. Let $\Lambda$ be a free linear system on $X$. Then there is a canonical isomorphism
	$$S(IC_X(\bV),\varphi)|_H\otimes\sO_X(H)|_H=S(IC_H(\bV|_{H\cap X^o}),\varphi|_H)$$
	for a general $H\in \Lambda$.
\end{prop}
\begin{proof}
	Let $\pi:\widetilde{X}\to X$ be a desingularization so that $E:=\pi^{-1}(X\backslash X^o)$ is a simple normal crossing divisor and $\pi$ is biholomorphic over $X^o$. Let $H$ be in a general position so that $\widetilde{H}:=\pi^\ast H$ is smooth and intersects transversally with every stratum of $E$. 
	By Proposition \ref{prop_L2ext_birational}, it suffices to show that 
	\begin{align}\label{align_adjunction_blowup}
	S(IC_{\widetilde{X}}(\pi^\ast\bV),\pi^\ast\varphi)|_{\widetilde{H}}\otimes\sO_{\widetilde{X}}(\widetilde{H})|_{\widetilde{H}}=S(IC_{\widetilde{H}}(\pi^\ast\bV|_{\pi^\ast (H\cap X^o)}),\pi^\ast\varphi|_{\widetilde{H}})
	\end{align}
	for a general $H\in \Lambda$. 
	
	Since the problem is local, we assume that $\widetilde{X}=\Delta^n$ is a polydisc with the standard holomorphic coordinates $z_1,\dots,z_n$ such that $E=\{z_1\cdots z_r=0\}$. Denote $E_i=\{z_i=0\}$, $\forall i=1,\dots,r$. Let $(\widetilde{v_1}, \dots, \widetilde{v_m})$ be an $L^2$-adapted frame of $R(IC_X(\bV))$ as in Proposition \ref{prop_adapted_frame}. Denote 
	$$\psi_i:=-\sum_{j=1}^r\alpha_{E_j}(\widetilde{v_i})\log|z_j|^2.$$
	By Proposition \ref{prop_key_est}, we see that 
	\begin{align*}
	S(IC_{\widetilde{X}}(\pi^\ast\bV),\pi^\ast\varphi)\simeq \omega_{\widetilde{X}}\otimes\bigoplus_{i=1}^m\sI(\pi^\ast\varphi+\psi_i)\widetilde{v_i},
	\end{align*}
	\begin{align*}
	S(IC_{\widetilde{H}}(\pi^\ast\bV|_{\pi^\ast (H\cap X^o)})),\pi^\ast\varphi|_{\widetilde{H}})\simeq \omega_{{\widetilde{H}}}\otimes\bigoplus_{i=1}^m\sI(\pi^\ast\varphi|_{\widetilde{H}}+\psi_i|_{\widetilde{H}})\widetilde{v_i}|_{\widetilde{H}}
	\end{align*}
	and
	$$\sI(\pi^\ast\varphi+\psi_i)|_{\widetilde{H}}\simeq \sI(\pi^\ast\varphi|_{\widetilde{H}}+\psi_i|_{\widetilde{H}})$$
    since $H$ is in a general position. 
Consequently, (\ref{align_adjunction_blowup}) is obtained.
\end{proof}
\section{$L^2$-Dolbeault resolution of multiplier $S$-sheaf}
The purpose of this section is to prove Theorem \ref{thm_main_resolution}. 
\begin{thm}\label{thm_main_local1}
	Let $X$ be a complex space of dimension $n$ and $ds^2$ a hermitian metric on a Zariski open subset $X^o\subset X_{\rm reg}$ with $\omega$ its fundamental form. Let $\bV=(\cV,\nabla,\cF^\bullet,h_\bV)$ be an $\bR$-polarized variation of Hodge structure on $X^o$. Let $(E,h_\varphi)$ be a holomorphic vector bundle on $X$ with a (possibly) singular hermitian metric $h_\varphi:=e^{-\varphi}h_0$. Assume that, locally at every point $x\in X$, there is a neighborhood $U$ of $x$, a strictly psh function $\lambda\in C^2(U)$ and a bounded psh function $\Phi\in C^2(U\cap X^o)$ such that $\sqrt{-1}\ddbar\lambda|_{U\cap X^o}\lesssim\omega|_{U\cap X^o}\lesssim\sqrt{-1}\ddbar\Phi$.
	Then the canonical map
	\begin{align}\label{align_fine_resolution}
	S(IC_X(\bV),\varphi)\otimes E\to \sD^{n,\bullet}_{X,ds^2}(S(\bV)\otimes E,h_\bV\otimes h_\varphi)
	\end{align}
	is a quasi-isomorphism. 
	If $X$ is moreover compact, then there is an isomorphism
	\begin{align*}
	H^q(S(IC_X(\bV),\varphi)\otimes E)\simeq H^{n,q}_{(2),\rm max}(X^o, S(\bV)\otimes E|_{X^o};ds^2,h_\bV\otimes h_\varphi),\quad \forall q.
	\end{align*}
\end{thm}
\begin{proof}
	By Lemma \ref{lem_Kernel} we have 
	$$S(IC_X(\bV),\varphi)\otimes E\simeq {\rm Ker}\left(\sD^{n,0}_{X,ds^2}(S(\bV)\otimes E,h_\bV\otimes h_\varphi)\stackrel{\dbar}{\to}\sD^{n,1}_{X,ds^2}(S(\bV)\otimes E,h_\bV\otimes h_\varphi)\right).$$
	
	It remains to show that (\ref{align_fine_resolution}) is exact at $\sD^{n,q}_{X,ds^2}(S(\bV)\otimes E,h_\bV\otimes h_\varphi)$ for $q>0$. Since the problem is local, we consider a point $x\in X$ and an open neighborhood $U$ of $x$ such that $E|_U\simeq\sO_U^{\oplus {\rm rk}E}$ and $h_0$ is quasi-isometric to the trivial metric $h_1$ on $\sO_U^{\oplus {\rm rk}E}$. Since
	$$\sqrt{-1}\Theta_{e^{-\varphi}h_1}(E|_{U\cap X^o})=\sqrt{-1}\ddbar\varphi\otimes{\rm Id}_E\geq c\sqrt{-1}\ddbar\lambda\otimes{\rm Id}_E\geq C\sqrt{-1}\ddbar\Phi\otimes{\rm Id}_E$$
	for some negative constants $c,C$.
	By assumptions, there is a constant $C'>0$ such that $$ (C'+C)\sqrt{-1}\partial\dbar\Phi\geq\omega|_{U\cap X^o}.$$ Let $h'=e^{-\varphi-C'\Phi}h_1$. The boundedness of $\Phi$ implies that  $h'\sim h_\varphi$.
	 It follows from Theorem \ref{thm_geq0_S(V)}  that 
	$$\sqrt{-1}\Theta_{h_\bV\otimes h'}(S(\bV)\otimes E|_{U\cap X^o})\geq C'\sqrt{-1}\partial\dbar\Phi\otimes{\rm Id}_{S(\bV)\otimes E}+\sqrt{-1}\Theta_{e^{-\varphi}h_1}(E|_{U\cap X^o})\otimes{\rm Id}_{S(\bV)}\geq \omega\otimes{\rm Id}_{S(\bV)\otimes E}$$
	holds on $U\cap X^o$.
	By Lemma \ref{lem_complete_metric_exists_locally} we may assume that $U\cap X^o$ admits a complete K\"ahler metric.
	Consequently, we have
	$$H^{n,q}_{(2)}(U\cap X^o, S(\bV)\otimes E|_{U\cap X^o};ds^2,h_\bV h_\varphi)=H^{n,q}_{(2)}(U\cap X^o, S(\bV)\otimes E|_{U\cap X^o};ds^2,h_\bV h')=0,\quad \forall q>0$$
	by Theorem \ref{thm_Hormander_incomplete}.
	This proves the exactness of (\ref{align_fine_resolution}) at $\sD^{n,q}_{X,ds^2}(S(\bV)\otimes E,h_\bV\otimes h_\varphi)$ for all $q>0$. 
	Since $\sD^{n,\bullet}_{X,ds^2_0}(E,h_\varphi)$ is a complex of fine sheaves (Lemma \ref{lem_fine_sheaf}), we obtain the second claim.
\end{proof}
A typical example of the metric $ds^2$ that satisfies the conditions in Theorem \ref{thm_main_local1} is the hermitian metric on $X$, which always exists by partition of unity. For applications we require $ds^2$ to be a complete K\"ahler metric.
Such kind of metric exists if $X$ is the truncation of a weakly pseudoconvex K\"ahler space.
Recall that a complex space is a weakly pseudoconvex (K\"ahler) space if it is a (K\"ahler) complex space that admits a smooth psh exhausted function.
\begin{lem}\label{lem_complete_Kahler_pseudoconvex}
	Let $X$ be a weakly pseudoconvex K\"ahler space with $\psi$ a smooth  psh exhausted function on $X$. Denote $X_c:=\{x\in X|\psi(x)<c\}$. Let $X^o\subset X_{\rm reg}$ be a Zariski open subset. Then, for every $c\in\bR$, there exists a complete K\"ahler metric $ds^2$ on $X_c\cap X^o$ satisfying that for every point $x\in X_c$, there is a neighborhood $U$ of $x$, a bounded function $\Phi\in C^{\infty}(U)$ and a smooth strictly psh function $\lambda$ on $U$ such that $\sqrt{-1}\ddbar\lambda\lesssim ds^2\sim\sqrt{-1}\ddbar\Phi$. 
\end{lem}
\begin{proof}
	The construction of the metric is the motivated by Ohsawa \cite[Lemma 2.6]{Ohsawa2018}. 
	Let $U$ be a neighborhood of a point $x\in X\backslash X^o$. Assume that $U\backslash X^o\subset U$ is defined by $f_1,\dots,f_r\in\sO_U(U)$. Let 
	$$\varphi_U:=\frac{1}{\log(-\log\sum_{i=1}^r|f_i|^2)}+\phi_U,$$
	where $\phi_U$ is a strictly $C^\infty$ psh function on $U$ so that $\varphi_U$ is strictly psh.
	Then the quasi-isometric class of $\sqrt{-1}\ddbar\varphi_U$ is independent of the choice of $\{f_1,\dots,f_r\}$ and $\phi_U$. By partition of unity, the potential functions $\varphi_U$ can be glued to a global function $\varphi$ (not necessarily psh) on $X$ so that 
	\begin{align*}
	\sqrt{-1}\ddbar\varphi|_U\sim\sqrt{-1}\ddbar\varphi_U
	\end{align*}
	near every point $x\in X\backslash X^o$ and $\varphi\equiv0$ away from a neighborhood of $X\backslash X^o$.
	
	Denote $u=-\log\sum_{i=1}^r|f_i|^2$. Then
	\begin{align*}
	\sqrt{-1}\ddbar\varphi|_U&\sim\sqrt{-1}\frac{2+\log u}{u^2\log^3u}\partial u\wedge\dbar u+\sqrt{-1}\frac{-\ddbar u}{u\log^2u}+\sqrt{-1}\ddbar\phi_U\\\nonumber
	&\sim\sqrt{-1}\frac{\partial u\wedge\dbar u}{u^2\log^2u}+\sqrt{-1}\frac{-\ddbar u}{u\log^2u}+\sqrt{-1}\ddbar\phi_U.
	\end{align*}
	Hence $\log\log u$ is a smooth psh  exhausted function near $X\backslash X^o$ such that 
	$$|d\log\log u|_{\sqrt{-1}\ddbar\varphi}\lesssim 2.$$
	By the Hopf-Rinow theorem, $\sqrt{-1}\ddbar\varphi$ is locally complete near $X\backslash X^o$.
	
	Let $c\in\bR$ and let $\omega_0$ be a K\"ahler hermitian metric on $X$. By adding a constant to $\psi$ we assume that $\psi\geq 0$. Then $\psi_c:=\psi+\frac{1}{c-\psi}$ is a smooth  psh exhausted function on $X_c=\{x\in X|\psi(x)<c\}$. 
	Hence, $\omega_0+\sqrt{-1}\ddbar\psi^2_c$
	is  a complete K\"ahler metric on $X_c$ (\cite[Theorem 1.3]{Demailly1982}).
	Since $\overline{X_c}$ is compact, $$\sqrt{-1}\ddbar\varphi+K(\omega_0+\sqrt{-1}\ddbar\psi^2_c),\quad K\gg0$$
	is positive definite and it provides the desired complete K\"ahler metric on $X_c\cap X^o$.
\end{proof}

\section{Vanishing theorems for $S$-sheaf}
Vanishing theorems and injectivity theorems for the $S$-sheaf are deduced from Theorem \ref{thm_main_local1} in this section.
\subsection{Nadel vanishing theorem}
\begin{thm}\label{thm_Nadel_vanishing}
	Let $X$ be a weakly pseudoconvex K\"ahler space and  $\omega$ a K\"ahler hermitian metric on $X$. Let $\bV$ be an $\bR$-polarized variation of Hodge structure on some Zariski open subset $X^o\subset X_{\rm reg}$ and $(L,h_\varphi)$ a holomorphic line bundle with a possibly singular hermitian metric $h_\varphi:=e^{-\varphi}h$. If $\sqrt{-1}\Theta_{h_\varphi}(L)\geq\epsilon\omega$ as currents for some $\epsilon>0$, then
	$$H^q(X,S(IC_X(\bV),\varphi)\otimes L)=0,\quad \forall q>0.$$
\end{thm}
\begin{proof}
    By Theorem \ref{thm_main_local1}
    there is a quasi-isomorphism
	$$S(IC_X(\bV),\varphi)\otimes L\simeq_{\rm q.i.s.}\sD^{n,\bullet}_{X,\omega}(S(\bV)\otimes L, h_\bV\otimes h_\varphi).$$
	 Lemma \ref{lem_fine_sheaf} implies that  $\sD^{n,\bullet}_{X,\omega}(S(\bV)\otimes L, h_\bV\otimes h_\varphi)$ is a complex of fine sheaves. Thus
	\begin{align*}
	H^q(X,S(IC_X(\bV),\varphi)\otimes L)\simeq H^q\left(\Gamma(X,\sD^{n,\bullet}_{X,\omega}(S(\bV)\otimes L, h_\bV\otimes h_\varphi))\right).
	\end{align*}
	Now let $q>0$ and let $\alpha\in \Gamma(X,\sD^{n,q}_{X,\omega}(S(\bV)\otimes L, h_\bV\otimes h_\varphi))$ be a locally $L^2$ form such that $\dbar\alpha=0$. We would like to show that there exists $\beta\in \Gamma(X,\sD^{n,q-1}_{X,\omega}(S(\bV)\otimes L, h_\bV\otimes h_\varphi))$ satisfying that  $\dbar\beta=\alpha$.
	
	Let $\psi$ be a smooth psh exhausted  function on $X$ and let $\chi$ be a convex increasing function which is of fast growth at infinity so that 
	\begin{align*}
	\int_{X^o}|\alpha|_{\omega,h_\bV\otimes h_\varphi}^2e^{-\chi\circ\psi}{\rm vol_\omega}<\infty.
	\end{align*}
	Let $h':=e^{-\chi\circ\psi}h_\varphi$. Consequently, we have 
	\begin{align*}
    \sqrt{-1}\Theta_{h_\bV\otimes h'}(S(\bV)\otimes L)&=\sqrt{-1}\Theta_{h_\bV}(S(\bV))+\sqrt{-1}\ddbar(\chi\circ\psi)\otimes {\rm Id}_{S(\bV)}+\sqrt{-1}\Theta_{h_\varphi}(L)\otimes {\rm Id}_{S(\bV)}\\\nonumber
    &\geq\epsilon\omega\otimes {\rm Id}_{S(\bV)}
    \end{align*}
    by Theorem \ref{thm_geq0_S(V)}.
    Moreover, Lemma \ref{lem_complete_Kahler_pseudoconvex} implies that  $X_c\cap X^o:=\{x\in X^o|\psi(x)<c\}$
    admits a complete K\"ahler metric  for every $c\in\bR$. 
    It follows from Theorem \ref{thm_Hormander_incomplete} that  for every $c\in \bR$ there is $$\beta_c\in L^{n,q-1}_{(2)}(X_c\cap X^o,S(\bV)\otimes L;\omega,h_\bV\otimes h')\subset \Gamma(X_c,\sD^{n,q-1}_{X,\omega}(S(\bV)\otimes L, h_\bV\otimes h_\varphi))$$
    such that $\dbar\beta_c=\alpha|_{X_c}$ and $\|\beta_c\|_{\epsilon\omega,h_\bV\otimes h'}\leq \frac{1}{q}\|\alpha\|_{\epsilon\omega,h_\bV\otimes h'}$. By taking a weak limit of $\beta_c$ we obtain a  solution of the equation $\alpha=\dbar\beta$ such that $\beta\in \Gamma(X,\sD^{n,q-1}_{X,\omega}(S(\bV)\otimes L, h_\bV\otimes h_\varphi))$.
\end{proof}

\begin{cor}\label{cor_KV_vanishing}
	Let $X$ be a projective algebraic variety of dimension $n$ and let $\bV$ be an $\bR$-polarized variation of Hodge structure on some Zariski open subset $X^o\subset X_{\rm reg}$. Let $L$ be a line bundle such that some positive multiple $mL=F+D$ where $F$ is a nef line bundle and $D$ is an effective divisor. Then
	$$H^q(X,S(IC_X(\bV),\frac{\varphi_D}{m})\otimes L)=0,\quad \forall q>n-{\rm nd}(F).$$
	Here $\varphi_D$ is the psh function associated to $D$ and
	$${\rm nd}(F)=\max\{k=1,\dots,n|c_1(F)^k\neq 0\in H^{2k}(X,\bR)\}.$$
\end{cor}
\begin{rmk}\label{rmk_varphi_D}
	Let $s_D\in \Gamma(X,\sO_X(D))$ and $h'_D$ a hermitian metric on $\sO_X(D)$. Let $\varphi_D:=\log|s_D|^2_{h'_D}$. Then the singular metric $h_D:=e^{-\varphi_D}h'_D$ is independent of the choice of $h'_D$ and $\sqrt{-1}\Theta_{h_D}(\sO_X(D))\geq 0$ as a current.
	Since $\varphi_D$ is differed by a smooth function for different choices of $s_D$ and $h'_D$, Lemma \ref{lem_S_change_varphi_bounded} shows that $S(IC_X(\bV),\frac{\varphi_D}{m})$ is independent of the choice of $s_D$ and $h'_D$.
\end{rmk}
\begin{proof}
	{\bf Case I: ${\rm nd}(F)=n$. }
	In this case $F$ is big and nef. Then $bF=A+E$ for some constant $b>0$, an ample line bundle $A$ and an effective Cartier divisor $E$. Let $h_A$ be a metric with positive curvature on $A$ and $h'_E$ a metric on $\sO_X(E)$. Let $h_E:=e^{-\varphi_E}h'_E$ (Remark \ref{rmk_varphi_D}). Then $h_{F,0}=(h_A\otimes h_E)^{\frac{1}{b}}$ is a singular metric such that
	\begin{align*}
	\sqrt{-1}\Theta_{h_{F,0}}(F)\geq \frac{1}{b}\omega.
	\end{align*}
	Here $\omega=\sqrt{-1}\Theta_{h_A}(A)$ is a hermitian metric on $X$. Since $F$ is nef, there is a metric $h_{F,\epsilon}$ such that $\sqrt{-1}\Theta_{h_{F,\epsilon}}(F)\geq-\epsilon\omega$ for every $\epsilon>0$. Let $h_D=e^{-\varphi_D}h'_D$ be the singular metric on $\sO_X(D)$ as in Remark \ref{rmk_varphi_D}. Define
	\begin{align*}
	h_L:=(h_{F,\epsilon}^{1-\delta}\otimes h_{F,0}^{\delta}\otimes h_D)^{\frac{1}{m}}=e^{-\frac{\delta\varphi_E}{bm}-\frac{\varphi_D}{m}}(h_{F,\epsilon}^{1-\delta}\otimes h_{A}^{\delta/b}\otimes h'^{\delta/b}_{E}\otimes h'_D)^{\frac{1}{m}},\quad 0<b\epsilon\ll\delta\ll 1.
	\end{align*}
	Then
	\begin{align*}
	\sqrt{-1}\Theta_{h_L}(L)&=\frac{1}{m}\left((1-\delta)\sqrt{-1}\Theta_{h_{F,\epsilon}}(F)+\delta\sqrt{-1}\Theta_{h_{F,0}}(F)+\sqrt{-1}\Theta_{h_D}(D)\right)\\\nonumber
	&\geq \frac{1}{m}\left(-(1-\delta)\epsilon\omega+\frac{\delta}{b}\omega\right)\geq\frac{\delta\epsilon}{m}\omega.
	\end{align*}
Consequently, Theorem \ref{thm_Nadel_vanishing} yields
$$H^q(X,S(IC_X(\bV),\varphi_L)\otimes L)=0,\quad \forall q>0,$$
where $\varphi_L=\frac{\delta\varphi_E}{bm}+\frac{\varphi_D}{m}$. 
Moreover, Proposition \ref{prop_approx_S} implies that $S(IC_X(\bV),\varphi_L)=S(IC_X(\bV),\frac{1}{m}\varphi_D)$ when $\delta>0$ is small enough. This proves the theorem under the condition that ${\rm nd}(F)=n$.

	{\bf Case II: ${\rm nd}(F)<n$.} Let $\pi:\widetilde{X}\to X$ be a desingularization so that $\pi$ is biholomorphic over $X^o\backslash D$ and the exceptional loci $E:=\pi^{-1}((X\backslash X^o)\cup D)$ is a simple normal crossing divisor. Let $H$ be a reduced ample hypersurface in a general position so that $\pi^{-1}H$ is smooth and has normal crossings with $E$. By taking $H$ sufficiently ample we assume that $F+mH$ is ample. 
	
	By Proposition \ref{prop_adjunction} we have 
	\begin{align}\label{align_adjunction_S}
	S(IC_X(\bV),\frac{1}{m}\varphi_D)|_H\otimes\sO_X(H)|_H\simeq S(IC_{H}(\bV|_{X^o\cap H}),\frac{1}{m}\varphi_{D\cap H})
	\end{align}
	when $H$ is in a general position. We further assume that $H$ contains no associated points of $S(IC_X(\bV),\frac{1}{m}\varphi_D)$. This implies that the sequence
	$$0\to S(IC_X(\bV),\frac{1}{m}\varphi_D)\otimes L\to S(IC_X(\bV),\frac{1}{m}\varphi_D)\otimes L\otimes\sO_X(H)\to S(IC_X(\bV),\frac{1}{m}\varphi_D)\otimes L\otimes\sO_X(H)\otimes \sO_H\to 0$$
	is exact.
 There is therefore a long exact sequence
	\begin{align}\label{align_long_exact_sequence}
	\cdots&\to H^q\left(X,S(IC_X(\bV),\frac{1}{m}\varphi_D)\otimes L\right)\to H^q\left(X,S(IC_X(\bV),\frac{1}{m}\varphi_D)\otimes L\otimes\sO_X(H)\right)\\\nonumber 
	&\to H^q\left(H,S(IC_X(\bV),\frac{1}{m}\varphi_D)|_H\otimes L|_H\otimes\sO_X(H)|_H\right)\to\cdots.
	\end{align}
	Since $F+mH$ is ample, ${\rm nd}(F+mH)=n$. Then we have
	\begin{align*}
	H^q\left(X,S(IC_X(\bV),\frac{1}{m}\varphi_D)\otimes L\otimes\sO_X(H)\right)=0,\quad \forall q>0
	\end{align*}
	by Case I.
It follows from (\ref{align_adjunction_S}) and (\ref{align_long_exact_sequence})that there are isomorphisms
	\begin{align}\label{align_KV_induction}
	H^q\left(X,S(IC_X(\bV),\frac{1}{m}\varphi_D)\otimes L\right)\simeq H^{q-1}\left(H,S(IC_{H}(\bV|_{X^o\cap H}),\frac{1}{m}\varphi_{D\cap H})\otimes L|_H\right),\quad \forall q>1.
	\end{align}
Subsequently, since ${\rm nd}(L|_H)={\rm nd}(L)$, we have 
\begin{align*}
H^{q}\left(H,S(IC_{H}(\bV|_{X^o\cap H}),\frac{1}{m}\varphi_{D\cap H})\otimes L|_H\right)=0,\quad \forall q>n-1-{\rm nd}(L|_H)
\end{align*}
by induction on $\dim X$.
Consequently, (\ref{align_KV_induction}) implies that 
\begin{align*}
H^q\left(X,S(IC_X(\bV),\frac{1}{m}\varphi_D)\otimes L\right)=0,\quad \forall q>n-{\rm nd}(L).
	\end{align*}
\end{proof}
\subsection{Relative vanishing theorem}
Let $f:X\to Y$ be a proper holomorphic map between complex spaces. A $(1,1)$-current $\alpha$ on $X$ is $f$-positive if, for every point $y\in Y$ there is a neighborhood $U$ of $y$, a hermitian metric $\omega_U$ on $U$ and a hermitian metric $\omega'$ on $f^{-1}(U)$ such that
$\alpha|_{f^{-1}U}+f^\ast\omega_U\geq \omega'$. 
\begin{cor}\label{cor_Nadel_rel}
	Let $f:X\to Y$ be a surjective proper K\"ahler holomorphic map between complex spaces. Let $\bV$ be an $\bR$-polarized variation of Hodge structure on some Zariski open subset $X^o\subset X_{\rm reg}$. Let $(L,h_\varphi)$ be a holomorphic line bundle on $X$ with a possibly singular hermitian metric $h_\varphi:=e^{-\varphi}h$. Assume that $\sqrt{-1}\Theta_{h_\varphi}(L)$ is $f$-positive. Then
	\begin{align*}
	R^if_\ast(S(IC_X(\bV),\varphi)\otimes L)=0,\quad \forall i>0.
	\end{align*}
\end{cor}
\begin{proof}
	Since the problem is local, we may assume that $Y$ admits a non-negative smooth  strictly psh exhausted function $\psi$ so that $\psi^{-1}\{0\}=\{y\}\subset Y$ is a point. To achieve this, one may embed $Y$ into $\bC^N$ as a closed Stein analytic subspace and take $\psi=|z|^2$. Moreover, since $f$ is proper, $f^\ast\psi$ is a smooth psh exhausted function on $X$. 
	
	By Lemma \ref{lem_complete_Kahler_pseudoconvex}, for every $c\in \bR$ there is a complete K\"ahler metric $\omega_c$ on $X^o\cap f^{-1}(Y_c)$ where $Y_c=\{y\in Y|\psi(y)<c\}$.
	Since $\sqrt{-1}\Theta_{h_\varphi}(L)$ is $f$-positive, we assume that
	\begin{align*}
	\sqrt{-1}\Theta_{e^{-C\psi}h_\varphi}(L)=\sqrt{-1}\Theta_{h_\varphi}(L)+Cf^\ast\left(\sqrt{-1}\ddbar\psi\right)\geq\omega'
	\end{align*}
	for some hermitian metric $\omega'$ on $f^{-1}(Y_c)$ and some constant $C>0$. It follows from Theorem \ref{thm_Nadel_vanishing} that
	\begin{align*}
	H^i(f^{-1}(Y_c),S(IC_X(\bV),\varphi)\otimes L)\simeq 	H^i(f^{-1}(Y_c),S(IC_X(\bV),\varphi+C\psi)\otimes L)=0,\quad \forall i>0.
	\end{align*}
	Here the isomorphism follows from the boundedness of $\psi$ (Lemma \ref{lem_S_change_varphi_bounded}). Taking the limit $c\to 0$ we see  that 
	\begin{align*}
	R^if_\ast(S(IC_X(\bV),\varphi)\otimes L)_y=0,\quad \forall i>0.
	\end{align*}
	This proves the corollary.
\end{proof}
\subsection{Kodaira-Nakano-Kazama type vanishing theorem}
\begin{thm}\label{thm_partial_vanishing}
Let $X$ be a weakly pseudoconvex K\"ahler space of dimension $n$. Let $\bV$ be an $\bR$-polarized variation of Hodge structure on some Zariski open subset $X^o\subset X_{\rm reg}$. Let $(E,h)$ be a hermitian vector bundle on $X$. Assume that $(S(\bV)\otimes E,h_\bV\otimes h)$ is $m$-Nakano positive. Then
$$H^q(X,S(IC_X(\bV))\otimes E)=0$$
whenever $q\geq1$ and $m\geq{\rm min}\{n-q+1,{\rm rk}S(\bV)+{\rm rk}E\}$.
\end{thm}
\begin{proof}
	By Theorem \ref{thm_main_local1} there is a quasi-isomorphism
	$$S(IC_X(\bV))\otimes E\simeq_{\rm q.i.s.}\sD^{n,\bullet}_{X,\omega}(S(\bV)\otimes E, h_\bV\otimes h)$$
	where $\omega$ is a K\"ahler hermitian metric on $X$.
	 Notice that $\sD^{n,\bullet}_{X,\omega}(S(\bV)\otimes E, h_\bV\otimes h)$ is a complex of fine sheaves by Lemma \ref{lem_fine_sheaf}. Hence
	\begin{align*}
	H^q(X,S(IC_X(\bV))\otimes E)\simeq H^q\left(\Gamma(X,\sD^{n,\bullet}_{X,\omega}(S(\bV)\otimes E, h_\bV\otimes h))\right).
	\end{align*}
	Now let $q>0$ such that $m\geq{\rm min}\{n-q+1,{\rm rk}S(\bV)+{\rm rk}E\}$. Let $\alpha\in \Gamma(X,\sD^{n,q}_{X,\omega}(S(\bV)\otimes E, h_\bV\otimes h))$ be a locally $L^2$ form such that $\dbar\alpha=0$. It suffices to find $\beta\in \Gamma(X,\sD^{n,q-1}_{X,\omega}(S(\bV)\otimes E, h_\bV\otimes h))$ so that $\dbar\beta=\alpha$.
	
	Denote $A=[\sqrt{-1}\Theta_{h_\bV\otimes h}(S(\bV)\otimes E),\Lambda_\omega]$.
	Let $\psi$ be a smooth psh exhausted  function on $X$ and $\chi$ a convex increasing function. Denote $h_{\chi}:=e^{-\chi\circ\psi}h$. Then
	\begin{align*}
	\sqrt{-1}\Theta_{h_\bV\otimes h_\chi}(S(\bV)\otimes E)&=\sqrt{-1}\Theta_{h_\bV\otimes h}(S(\bV)\otimes E)+\sqrt{-1}\ddbar(\chi\circ\psi)\otimes {\rm Id}_{S(\bV)\otimes E}\\\nonumber
	&\geq \sqrt{-1}\Theta_{h_\bV\otimes h}(S(\bV)\otimes E).
	\end{align*}
	Denote  $A_\chi=[\sqrt{-1}\Theta_{h_\bV\otimes h_\chi}(S(\bV)\otimes E),\Lambda_\omega]$. Then $A_\chi\geq A>0$ in bidegree $(n,q)$. Thus the integrals
	$$\int_{X^o}|\alpha|_{\omega,h_\chi}^2{\rm vol_\omega}=\int_{X^o}|\alpha|_{\omega,h}^2e^{-\chi\circ\psi}{\rm vol_\omega}$$
	and
	$$\int_{X^o}(A^{-1}_\chi\alpha,\alpha)_{\omega,h_\chi}^2{\rm vol_\omega}\leq\int_{X^o}(A^{-1}\alpha,\alpha)_{\omega,h}^2e^{-\chi\circ\psi}{\rm vol_\omega}$$
	are finite if $\chi$ grows fast enough at infinity. By Lemma \ref{lem_complete_Kahler_pseudoconvex}, $$X_c\cap X^o:=\{x\in X^o|\chi\circ\psi(x)<c\}$$
	admits a complete K\"ahler metric for every $c\in\bR$.
Subsequently, \cite[Theorem 5.1]{Demailly2012} implies that  there exists $$\beta_c\in L^{n,q-1}_{(2)}(X_c\cap X^o,S(\bV)\otimes E;\omega,h_\bV\otimes h_\chi)\subset \Gamma(X_c,\sD^{n,q-1}_{X,\omega}(S(\bV)\otimes E, h_\bV\otimes h))$$
	such that $\dbar\beta_c=\alpha|_{X_c}$ and $$\|\beta_c\|_{\omega,h_\bV\otimes h_\chi}\leq \frac{1}{q}\int_{X^o}(A^{-1}\alpha,\alpha)_{\omega,h}^2e^{-\chi\circ\psi}{\rm vol_\omega}.$$ By taking a weak limit of a certain subsequence of $\{\beta_c\}$ we obtain the solution of the equation $\alpha=\dbar\beta$ such that $\beta\in \Gamma(X,\sD^{n,q-1}_{X,\omega}(S(\bV)\otimes E, h_\bV\otimes h))$.	
	This proves the theorem.
\end{proof}
\subsection{Enoki-Koll\'ar type injectivity theorem}
\begin{thm}\label{thm_Enoki_Kollar_inj}
	Let $X$ be a compact K\"ahler space of dimension $n$. Let $\bV$ be an $\bR$-polarized variation of Hodge structure on some Zariski open subset $X^o\subset X_{\rm reg}$. Let $(L,h_{\varphi_L})$ and $(F,h_{\varphi_F})$ be holomorphic line bundles on $X$ with singular hermitian metrics $h_{\varphi_L}:=e^{-\varphi_L}h$ and $h_{\varphi_F}:=e^{-\varphi_F}h'$ which are smooth on a Zariski open subset $U\subset X$. Assume the following conditions.
	\begin{enumerate}
		\item $\sqrt{-1}\Theta_{h_{\varphi_F}}(F)\geq0$ on $U$.
		\item $\sqrt{-1}( \Theta_{h_{\varphi_F}}(F)-\epsilon  \Theta_{h_{\varphi_L}}(L))\geq0$ on $U$ for some positive constant $\epsilon$.
	\end{enumerate}
	Let $s\in \Gamma(X,L)$ be a nonzero holomorphic section such that $\sup|s|_{h_{\varphi_L}}<\infty$. Then the multiplication homomorphism 
	$$\times s:H^q(X,S(IC_X(\bV),\varphi_F)\otimes F)\to H^q(X,S(IC_X(\bV),\varphi_F+\varphi_L)\otimes F\otimes L)$$
	is injective for every integer $q\geq0$.
\end{thm}
\begin{proof}
	The proof is parallel to the arguments in \cite{Fujino2012}. 
	By virtue of Lemma \ref{lem_complete_Kahler_pseudoconvex} and Theorem  \ref{thm_main_local1}, there exists a complete K\"ahler metric $\omega$ on $V:=X^o\cap U$ such that the canonical maps
	\begin{align*}
	H^q(X,S(IC_X(\bV),\varphi_F)\otimes F)\to H^{n,q}_{(2)}(V,S(\bV)\otimes F;\omega,h_\bV\otimes h_{\varphi_F})
	\end{align*}
	and
	\begin{align*}
	H^q(X,S(IC_X(\bV),\varphi_F+\varphi_L)\otimes F\otimes L)\to H^{n,q}_{(2)}(V,S(\bV)\otimes F\otimes L;\omega,h_\bV\otimes h_{\varphi_F}\otimes h_{\varphi_L})
	\end{align*}
	are isomorphisms for every $q\geq0$.
	Since $S(IC_X(\bV),\varphi_F)$ and $S(IC_X(\bV),\varphi_F+\varphi_L)$ are coherent, $H^{n,q}_{(2)}(V,S(\bV)\otimes F;\omega,h_\bV\otimes h_{\varphi_F})$ and $H^{n,q}_{(2)}(V,S(\bV)\otimes F\otimes L;\omega,h_\bV\otimes h_{\varphi_F}\otimes h_{\varphi_L})$ are finite dimensional for each $q\geq0$. Thus there are isomorphisms
	\begin{align*}
	H^{n,q}_{(2)}(V,S(\bV)\otimes F;\omega,h_\bV\otimes h_{\varphi_F})\simeq \sH^{n,q}_{(2)}(S(\bV)\otimes F)
	\end{align*}
	and 
	\begin{align*}
	H^{n,q}_{(2)}(V,S(\bV)\otimes F\otimes L;\omega,h_\bV\otimes h_{\varphi_F}\otimes h_{\varphi_L})\simeq \sH^{n,q}_{(2)}(S(\bV)\otimes F\otimes L),
	\end{align*}
	where 
	\begin{align*}
\sH^{n,q}_{(2)}(S(\bV)\otimes F)=\{\alpha\in D^{n,q}_{\rm max}(V,S(\bV)\otimes F;\omega,h_\bV\otimes h_{\varphi_F})|\dbar\alpha=0,\dbar^\ast\alpha=0\}
	\end{align*}
	and
\begin{align*}
\sH^{n,q}_{(2)}(S(\bV)\otimes F\otimes L)=\{\alpha\in D^{n,q}_{\rm max}(V,S(\bV)\otimes F\otimes L;\omega,h_\bV\otimes h_{\varphi_F}\otimes h_{\varphi_L})|\dbar\alpha=0,\dbar^\ast\alpha=0\}.
\end{align*}	
We claim that the multiplication map $$\times s:\sH^{n,q}_{(2)}(S(\bV)\otimes F)\longrightarrow \sH^{n,q}_{(2)}(S(\bV)\otimes F\otimes L)$$
is well-defined.
If the claim is true, the theorem is obvious. Assume that $su=0$ in $\sH^{n,q}_{(2)}(S(\bV)\otimes F\otimes L)$. Since $s$ is holomorphic over $X$, the locus $\{s\neq 0\}$ is dense in $V$. Hence $u=0$ for $u \in \sH^{n,q}_{(2)}(S(\bV)\otimes F)$ since $u$ is smooth over $V$. This implies the desired injectivity. Thus it is sufficient to prove the above claim.

	Take an arbitrary $u\in\sH^{n,q}_{(2)}(S(\bV)\otimes F)$. Since $\sup|s|_{h_{\varphi_L}}<\infty$, $\|u\|<\infty$ and $\bar{\partial}u=0$, we obtain that 
	$$\|su\| \leq \sup|s|_{h_{\varphi_L}}\|u\| <\infty,\quad \bar{\partial}(su)=0.$$
	By the Bochner-Kodaria-Nakano identity,
	\begin{align*}
	\Delta_{\dbar}u=\Delta_{D'}u +[\sqrt{-1}\Theta_{h_\bV\otimes h_{\varphi_F}}(S(\bV)\otimes F),\Lambda]u,
	\end{align*}
	where $\Lambda$ is the adjoint of $\omega \wedge \cdot$ and $\nabla=D'+\dbar$ is the bidegree decomposition of the Chern connection $\nabla$ associated to $h_\bV\otimes h_{\varphi_F}$.
Since $u\in\sH^{n,q}_{(2)}(S(\bV)\otimes F)$,
	we have $\Delta_{\dbar} u=0$. So 
	\begin{align}\label{align_Bochner_ineq}
	\Delta_{D'}u+[\sqrt{-1}\Theta_{h_\bV\otimes h_{\varphi_F}}(S(\bV)\otimes F),\Lambda]u=0.
	\end{align}
By Assumption (1) and Theorem \ref{thm_geq0_S(V)}, we obtain that
	\begin{align*}
	\langle[\sqrt{-1}\Theta_{h_\bV\otimes h_{\varphi_F}}(S(\bV)\otimes F),\Lambda]u,u\rangle_{h_\bV\otimes h_{\varphi_F}}\geq 0,
	\end{align*}
	where $\langle,\rangle_{h_\bV\otimes h_{\varphi_F}}$ is the pointwise inner product with respect to $h_\bV\otimes h_{\varphi_F}$ and $\omega$. 
	Since $\omega$ is complete, we obtain that
	\begin{align}\label{align_ineq_laplace}
	\langle \Delta_{D'}u,u\rangle=\|D'u\|^2+\|D'^{*}u\|^2\geq0.
	\end{align}
	By combining (\ref{align_Bochner_ineq}) with (\ref{align_ineq_laplace}), we obtain that 
	$$\|D'u\|^2=\|D'^{*}u\|^2=\langle\sqrt{-1}\Theta_{h_\bV\otimes h_{\varphi_F}}(S(\bV)\otimes F)\Lambda u,u\rangle_{h_\bV\otimes h_{\varphi_F}}=0,$$
	which implies  that 
	\begin{align}\label{align_injectivity_1}
	D'^{*}u=0\textrm{ and }\langle\sqrt{-1}\Theta_{h_\bV\otimes h_{\varphi_F}}(S(\bV)\otimes F)\Lambda u,u\rangle_{h_\bV\otimes h_{\varphi_F}}=0.
	\end{align}	
	Since $u\in \sH^{n,q}_{(2)}(S(\bV)\otimes F)$ and $s$ is holomorphic over $X$, (\ref{align_injectivity_1}) yields
	\begin{align*}
	D'^{*}(su)=-\ast \bar{\partial} \ast (su)=sD'^{*}u=0
	\end{align*}
	where $\ast$ is the Hodge star operator with respect to $\omega$.
	Moreover, due to the degree, we have $D'(su)=0$. As a result, $\Delta_{D'}(su)=0$.
	Applying the Bochner-Kodaria-Nakano identity to $su$, we get that  $$0\leq \|\dbar^{*}(su)\|^2\leq\langle\Delta_{\dbar}(su),su \rangle=\langle\sqrt{-1}\Theta_{h_\bV\otimes h_{\varphi_F}\otimes h_{\varphi_L}}(S(\bV)\otimes F\otimes L)\Lambda su,su\rangle_{h_\bV\otimes h_{\varphi_F}\otimes h_{\varphi_L}}.$$
	
    By the assumptions and Theorem \ref{thm_geq0_S(V)} we obtain that 
	\begin{align*}
	0\leq&\langle\sqrt{-1}\Theta_{h_\bV\otimes h_{\varphi_F}\otimes h_{\varphi_L}}(S(\bV)\otimes F\otimes L)\Lambda su,su\rangle_{h_\bV\otimes h_{\varphi_F}\otimes h_{\varphi_L}}\\\nonumber
	\leq& (1+\frac{1}{\epsilon}) |s|^2_{h_{\varphi_L}}\langle\sqrt{-1}\Theta_{h_\bV\otimes h_{\varphi_F}}(S(\bV)\otimes F)\Lambda u,u\rangle_{h_\bV\otimes h_{\varphi_F}}= 0.
	\end{align*}	
	Thus $\dbar^{*}(su)=0$. Hence $\Delta_{\dbar}(su)=0$, equivalently, $su\in \sH^{n,q}_{(2)}(S(\bV)\otimes F\otimes L)$. The proof of the claim is finished. 
\end{proof}
Assume that $F$ is a semi-ample holomorphic line bundle such that $H^0(X,F^{\otimes k}\otimes L^{-1})\neq0$ for some $k>0$. Denote $B=\{s=0\}$ and take a nonzero divisor $D'\in |F^{\otimes k}\otimes L^{-1}|$. Since $F$ is semi-ample, there is a smooth reduced divisor $D$ so that $B+D+D'\in |F^{\otimes m}|$ where $m$ is large enough so that $S(IC_X(\bV),\frac{\varphi_{D+D'}}{m})=S(IC_X(\bV))$ (Proposition \ref{prop_approx_S}). Let $h_B$ be the singular hermitian metric on $L$ associated to $B$ and $h_{D+D'}$ the singular hermitian metric on $\sO_X(D+D')$ associated to the divisor $D+D'$ (Remark \ref{rmk_varphi_D}). Then $U=X_{\rm reg}\backslash(B\cup D\cup D')$, $h_F:=(h_Bh_{D+D'})^{\frac{1}{m}}$ and $h_B$ satisfy the conditions in Theorem \ref{thm_Enoki_Kollar_inj}. 
\begin{cor}[Esnault-Viehweg type injectivity theorem]\label{cor_wulei}
	Let $X$ be a compact K\"ahler space and let $\bV$ be an $\bR$-polarized variation of Hodge structure on some Zariski open subset $X^o\subset X_{\rm reg}$. Let $(L,h_L)$ be a holomorphic line bundle on $X$ and $F$ a semi-ample holomorphic line bundle such that $H^0(X,F^{\otimes k}\otimes L^{-1})\neq0$ for some $k>0$.
	Let $s\in \Gamma(X,L)$ be a nonzero holomorphic section. Then the multiplication homomorphism 
	$$\times s:H^q(X,S(IC_X(\bV))\otimes F)\to H^q(X,S(IC_X(\bV))\otimes F\otimes L)$$
	is injective for every integer $q\geq0$.
\end{cor}
\section{Applications in the relative Fujita conjecture}
Some applications of the Nadel type vanishing theorem (Theorem \ref{thm_Nadel_vanishing}) are given on the separation of jets of adjoint bundles in the relative case. Before that we recall the singular metrics constructed by  Angehrn-Siu \cite{Siu1995} and Demailly \cite{Demailly2012}.
\begin{prop}[Angehrn-Siu \cite{Siu1995}]\label{prop_Siu}
	Let $X$ be a smooth projective algebraic variety and $L$ an ample line bundle on $X$. Assume that there is a rational number $\kappa>0$ such that $$L^k\cdot W\geq\left(\frac{1}{2}n(n+2r-1)+\kappa\right)^d$$
	for any irreducible subvariety $W$ of dimension $0\leq d\leq n$ in $X$. Let $x_1,\dots,x_r\in X$ and denote $m_0=\frac{1}{2}n(n+2r-1)$. Then there is a rational number $0<\epsilon<\kappa$ and a singular hermitian metric $h'$ on $L^{\otimes\frac{m_0+\epsilon}{m_0+\kappa}}$ with analytic singularities such that
	\begin{enumerate}
		\item $\sqrt{-1}\Theta_{h'}(L^{\otimes\frac{m_0+\epsilon}{m_0+\kappa}})\geq 0$;
		\item $x_1,\dots,x_r\in{\rm supp}\sO_X/\sI(h')$ and $x_1$ is an isolated point of ${\rm supp}\sO_X/\sI(h')$. Moreover there is a neighborhood $U$ of $x_1$ so that $h'$ is smooth on $U\backslash\{x_1\}$.
	\end{enumerate}
\end{prop}
\begin{proof}
	This is statement $(*)_0$ on page 299 in \cite{Siu1995} which is proved in Lemma 9.2 in loc. cit.
\end{proof}
\begin{prop}(Demailly \cite[Theorem 7.4]{Demailly2012})\label{prop_Demailly}
	Let $X$ be a smooth projective algebraic variety and $L$ an ample line bundle on $X$. Let $\omega$ be a K\"ahler form on $X$. Let $x_1,\dots,x_r\in X$ and let $s_1,\dots,s_r\in\bN$ be non-negative integers. Denote $m_0=2+\sum_{1\leq j\leq r}\binom{3n+2s_j-1}{n}$. Then there is a singular hermitian metric $h'$ on $\omega_X\otimes L^{\otimes m_0}$ with analytic singularities such that
	\begin{enumerate}
		\item $\sqrt{-1}\Theta_{h'}(\omega_X\otimes L^{\otimes m_0})\geq\epsilon\omega$ for some $\epsilon>0$.
		\item The singular loci of $h'$ is isolated and the weight $\varphi$ of $h'$ ($h'=e^{-\varphi} h_0$ for some smooth hermitian metric $h_0$) satisfies that 
		$$\nu(\varphi,x_j)\geq n+s_j,\quad j=1,\dots,r.$$
	\end{enumerate}
\end{prop}

\begin{thm}\label{thm_Hodge_Siu}
	Let $X$ be a projective $n$-fold and $D$ a (possibly empty) normal crossing divisor on $X$. Assume that $\bV$ is an $\bR$-polarized variation of Hodge structure on $X^o:=X\backslash D$. Let $L$ be an ample line bundle on $X$. Assume that there is a positive number $\kappa>0$ such that $$L^k\cdot W\geq\left(\frac{1}{2}n(n+2r-1)+\kappa\right)^d$$
	for any irreducible subvariety $W$ of dimension $0\leq d\leq n$ in $X$. Then the global holomorphic sections of $S(IC_X(\bV))\otimes L$ separate any set of $r$ distinct points $x_1,\dots, x_r\in X$, i.e. there is a surjective map
	$$H^0(X,S(IC_X(\bV))\otimes L)\to \bigoplus_{1\leq k\leq r}S(IC_X(\bV))\otimes L\otimes \sO_{X,x_k}/m_{X,x_k}.$$
\end{thm}
\begin{proof}
	By induction on $r$, the canonical morphism
	$$H^0(X,S(IC_X(\bV))\otimes L)\to \bigoplus_{2\leq k\leq r}S(IC_X(\bV))\otimes L\otimes \sO_{X,x_k}/m_{X,x_k}$$
	is surjective. It is therefore sufficient to show that, for every $v\in S(IC_X(\bV))\otimes L\otimes \sO_{X,x_1}/m_{X,x_1}$, there is a section $s_1\in H^0(X,S(IC_X(\bV))\otimes L)$ such that $s_1(x_1)=v$ and $s_1(x_k)=0$, $k=2,\dots,r$.
	
	By Proposition \ref{prop_Siu}, there is a singular hermitian metric $h'$ on $L^{\otimes\frac{m_0+\epsilon}{m_0+\kappa}}$ so that
	\begin{enumerate}
		\item $h'$ has analytic singularities;
		\item $\sqrt{-1}\Theta_{h'}(L^{\otimes\frac{m_0+\epsilon}{m_0+\kappa}})\geq 0$;
		\item $x_1,\dots,x_r\in{\rm supp}\sO_X/\sI(h')$ and $x_1$ is an isolated point of ${\rm supp}\sO_X/\sI(h')$.
	\end{enumerate}
	Let $h''$ be a smooth hermitian metric on $L^{\otimes\frac{\kappa-\epsilon}{m_0+\kappa}}$ such that $\sqrt{-1}\Theta_{h''}(L^{\otimes\frac{\kappa-\epsilon}{m_0+\kappa}})>0$. Let $h=h'h''$. Then
	\begin{align}\label{align_Siu_curvature_est}
	\sqrt{-1}\Theta_{h}(L)\geq\sqrt{-1}\Theta_{h''}(L^{\otimes\frac{\kappa-\epsilon}{m_0+\kappa}})>0.
	\end{align}
	Let $\varphi$ be the weight of $h$, i.e. $h=e^{-\varphi}h_0$ for some smooth hermitian metric $h_0$. Let $\varphi'$ be a quasi-psh function so that 
	\begin{enumerate}
		\item $\varphi'$ is smooth on $X\backslash\{x_2,\dots,x_r\}$.
		\item $e^{-\varphi'}$ is not locally integrable at any of $x_2,\cdots, x_r$.
		\item $\varphi'\geq\varphi+C$ for some $C\in\bR$.
	\end{enumerate}
    Such $\varphi'$ can be constructed as follows. Since $\varphi$ has analytic singularities, we assume that there is an open subset $U_1$ such that $\overline{U_1}\cap\{x_1,\dots,x_r\}=\{x_1\}$ and
    $\varphi=a\log(\sum_{j=1}^s|g_j|^2)+\lambda$ for some holomorphic functions $g_1,\dots,g_s$ on $U_1$, a constant $a>0$ and a bounded function $\lambda$. Since $x_1,\dots,x_r\in{\rm supp}\sO_X/\sI(h)$, $\{x_1\}\subset \{g_1=\cdots=g_s=0\}$. Let $\varphi'_1=a\log(\sum_{j=1}^l|g_j|^2)+\lambda$ where $g_{s+1},\dots,g_l$ are holomorphic functions such that $\{g_1=\cdots=g_l=0\}=\emptyset$. Let $U'$ be an open neighborhood of $\{x_2,\dots,x_r\}$ such that $\overline{U_1}\cap\overline{U'}=\emptyset$. Choose a smooth extension $\varphi'\in C^{\infty}(X\backslash\{x_2,\dots,x_r\})$ of $\varphi'_1\in C^\infty(U_1)$ and $\varphi\in C^\infty(U'\backslash\{x_2,\dots,x_r\})$. Then $\varphi'$ is the desired function.
    
    Consider the short exact sequence of sheaves
	\begin{align*}
	0\to S(IC_X(\bV),\varphi)\otimes L\to S(IC_X(\bV),\varphi')\otimes L\to S(IC_X(\bV),\varphi')/S(IC_X(\bV),\varphi)\otimes L\to 0.
	\end{align*}
	Taking the cohomologies we obtain the exact sequence
	\begin{align*}
	H^0(X,S(IC_X(\bV),\varphi')\otimes L)\to H^0(X,S(IC_X(\bV),\varphi')/S(IC_X(\bV),\varphi)\otimes L)\to H^1(X,S(IC_X(\bV),\varphi)\otimes L).
	\end{align*}
	By (\ref{align_Siu_curvature_est}) and Theorem \ref{thm_Nadel_vanishing} we have
	\begin{align*}
	H^1(X,S(IC_X(\bV),\varphi)\otimes L)=0.
	\end{align*}
	So the canonical map
	\begin{align}\label{align_siu_surject1}
	H^0(X,S(IC_X(\bV),\varphi')\otimes L)\to H^0(X,S(IC_X(\bV),\varphi')/S(IC_X(\bV),\varphi)\otimes L)
	\end{align}
	is surjective.
	
	Now let us investigate the structure of $S(IC_X(\bV),\varphi')/S(IC_X(\bV),\varphi)$ at $x_1$. Let $z_1,\dots,z_n$ be holomorphic coordinates centered at $x_1$ such that $D=\{z_1\cdots z_s=0\}$. Let $U\subset X$ denote the domain of the coordinate and denote $D_i=\{z_i=0\}$, $\forall i=1,\dots,s$. Let $\widetilde{v_1}, \dots, \widetilde{v_m}$ be an $L^2$-adapted frame of $R(IC_X(\bV))$ as in Proposition \ref{prop_adapted_frame}. Denote 
	$$\psi_i:=-\sum_{j=1}^s\alpha_{E_j}(\widetilde{v_i})\log|z_j|^2.$$
	By Proposition \ref{prop_key_est}, we see that 
	\begin{align*}
	S(IC_{X}(\bV),\varphi)|_U\simeq \omega_{U}\otimes\bigoplus_{i=1}^m\sI(\varphi+\psi_i)|_U\widetilde{v_i}
	\end{align*}
	and
	\begin{align*}
	S(IC_{X}(\bV),\varphi')|_U\simeq \omega_{U}\otimes\bigoplus_{i=1}^m\sI(\varphi'+\psi_i)|_U\widetilde{v_i}.
	\end{align*}
	Thus 
	\begin{align}\label{align_siu_cokernel}
	S(IC_{X}(\bV),\varphi')/S(IC_{X}(\bV),\varphi)|_U\simeq \omega_{U}\otimes\bigoplus_{i=1}^m\sI(\varphi'+\psi_i)/\sI(\varphi+\psi_i)|_U\widetilde{v_i}.
	\end{align}
	After a possible shrinking of $U$ we assume that $\varphi$ is smooth on $U\backslash\{x_1\}$. Recall that 
	$\alpha_{E_j}(\widetilde{v_i})\in(-1,0]$
	for every $j=1,\cdots,s$ in Proposition \ref{prop_adapted_frame}. Then $e^{-\psi_1}, \dots, e^{-\psi_m}$ are locally integrable on $U$. This implies that
	$$\sI(\varphi+\psi_i)_x\simeq\sO_{U,x},\quad i=1,\dots,m,\quad x\in U\backslash\{x_1\}.$$
	As a consequence
	\begin{align}\label{align_siu_supp1}
	U\cap{\rm supp}\sI(\varphi'+\psi_i)/\sI(\varphi+\psi_i)\subset\{x_1\}.
	\end{align}
	Since $\varphi'$ is smooth at $x_1$ and $e^{-\psi_1}, \dots, e^{-\psi_m}$ are locally integrable on $U$, one gets that 
	\begin{align*}
	\sI(\varphi'+\psi_i)_{x_1}\simeq \sO_{U,x_1},\quad i=1,\dots,m.
	\end{align*}
	Since $x_1\in{\rm supp}\sO_U/\sI(\varphi)$, one sees that
	\begin{align}\label{align_siu_supp3}
	\sI(\varphi+\psi_i)_{x_1}\subset \sI(\varphi)_{x_1}\subset m_{U,x_1},\quad i=1,\dots,m.
	\end{align}
	
	By combining (\ref{align_siu_cokernel}) with (\ref{align_siu_supp1}) and (\ref{align_siu_supp3}), we obtain a canonical surjective map
	$$H^0(X,S(IC_{X}(\bV),\varphi')/S(IC_{X}(\bV),\varphi)\otimes L)\to S(IC_X(\bV))\otimes L\otimes\sO_{X,x_1}/m_{X,x_1}.$$
	Since (\ref{align_siu_surject1}) is surjective, we get a surjective map
	$$H^0(X,S(IC_{X}(\bV),\varphi')\otimes L)\to S(IC_X(\bV))\otimes L\otimes\sO_{X,x_1}/m_{X,x_1}.$$
	For every $v\in S(IC_X(\bV))\otimes L\otimes \sO_{X,x_1}/m_{X,x_1}$, there is a section $$s_1\in H^0(X,S(IC_{X}(\bV),\varphi')\otimes L)$$ such that $s_1(x_1)=v$. Since $\varphi'$ is not locally integrable at any of $x_2,\dots,x_r$, we know  that $s_1(x_i)=0$, $i=2,\dots,r$.
    This proves the theorem.
\end{proof}
\begin{thm}\label{thm_sepjets}
	Let $X$ be a smooth projective algebraic variety and $D$ a (possibly empty) normal crossing divisor of $X$. Let $\bV$ be an $\bR$-polarized variation of Hodge structure on $X^o:=X\backslash D$. Let $L$ be an ample line bundle and $G$ a nef line bundle on $X$. Then $S(IC_X(\bV))\otimes \omega_X\otimes L^{\otimes m}\otimes G$ generates simultaneous jets of order $s_1,\dots, s_r\in \bN$ at arbitrary points $x_1,\dots, x_r\in X$, i.e. there is a surjective map
	$$H^0(X,S(IC_X(\bV))\otimes \omega_X\otimes L^{\otimes m}\otimes G)\to \bigoplus_{1\leq j\leq r}S(IC_X(\bV))\otimes \omega_X\otimes L^{\otimes m}\otimes G\otimes \sO_{X,x_k}/m^{s_k+1}_{X,x_k},$$
	provided that $m\geq 2+\sum_{1\leq k\leq r}\binom{3n+2s_k-1}{n}$. In particular, $S(IC_X(\bV))\otimes \omega_X\otimes L^{\otimes m}\otimes G$ is globally generated for $m\geq 2+\binom{3n-1}{n}$.
\end{thm}
\begin{proof}
	Denote $m_0=2+\sum_{1\leq k\leq p}\binom{3n+2s_k-1}{n}$. Let $h'$ be the singular hermitian metric on $\omega_X\otimes L^{\otimes m_0}$ as constructed in Proposition \ref{prop_Demailly}. Since $L$ and $G$ are nef, there is a smooth hermitian metric $h''$ on $L^{\otimes(m-m_0)}\otimes G$ such that $$\sqrt{-1}\Theta_{h''}(L^{\otimes(m-m_0)}\otimes G)\geq -\frac{\epsilon}{2}\omega$$
	where $\omega$ is a K\"ahler form on $X$.
	
	Let $h:=h'h''$ be the singular hermitian metric on $A:=\omega_X\otimes L^{\otimes m}\otimes G$. Then
	\begin{enumerate}
		\item $\sqrt{-1}\Theta_{h}(A)\geq\frac{\epsilon}{2}\omega$,
		\item ${\rm supp}(\sO_X/\sI(h))$ is $0$-dimensional and the weight $\varphi$ of $h$ satisfies
		$$\nu(\varphi,x_k)\geq n+s_k,\quad k=1,\dots,r.$$
	\end{enumerate}
    Consider the short exact sequence of sheaves
    \begin{align*}
    0\to S(IC_X(\bV),\varphi)\otimes A\to S(IC_X(\bV))\otimes A\to S(IC_X(\bV))/S(IC_X(\bV),\varphi)\otimes A\to 0.
    \end{align*}
    Taking the cohomology of this sequence, we obtain the exact sequence
    \begin{align*}
    H^0(X,S(IC_X(\bV))\otimes A)\to H^0(X,S(IC_X(\bV))/S(IC_X(\bV),\varphi)\otimes A)\to H^1(X,S(IC_X(\bV),\varphi)\otimes A).
    \end{align*}
    Since $\sqrt{-1}\Theta_{h}(A)\geq\frac{\epsilon}{2}\omega$, we have
    \begin{align*}
    H^1(X,S(IC_X(\bV),\varphi)\otimes A)=0
    \end{align*}
    by Theorem \ref{thm_Nadel_vanishing}.
    So the canonical map
    \begin{align}\label{align_siu_surject12}
    H^0(X,S(IC_X(\bV))\otimes A)\to H^0(X,S(IC_X(\bV))/S(IC_X(\bV),\varphi)\otimes A)
    \end{align}
    is surjective.
    
    Now let us investigate the structure of $S(IC_X(\bV))/S(IC_X(\bV),\varphi)$ at $x_k$, $k=1,\dots,r$. 
    
  Let $z_1,\dots,z_n$ be holomorphic coordinates on a neighborhood $U_k$ of $x_k$, centered at $x_k$, such that $D=\{z_1\cdots z_s=0\}$. Denote $D_i=\{z_i=0\}, \forall i=1,\dots,s$. Let $\widetilde{v_1}, \dots, \widetilde{v_m}$ be an $L^2$-adapted frame of $R(IC_X(\bV))$ at $x_k$ as in Proposition \ref{prop_adapted_frame}. Denote 
    $$\psi_i:=-\sum_{j=1}^s\alpha_{E_j}(\widetilde{v_i})\log|z_j|^2.$$
    It follows from Proposition \ref{prop_key_est} that 
    \begin{align*}
    S(IC_{X}(\bV),\varphi)|_{U_k}\simeq \omega_{U_k}\otimes\bigoplus_{i=1}^m\sI(\varphi+\psi_i)|_{U_k}\widetilde{v_i}.
    \end{align*}
    Hence
    \begin{align}\label{align_siu_cokernel2}
    S(IC_{X}(\bV))/S(IC_{X}(\bV),\varphi)|_{U_k}\simeq \omega_{U_k}\otimes\bigoplus_{i=1}^m\sO_{X}/\sI(\varphi+\psi_i)|_{U_k}\widetilde{v_i}.
    \end{align}
    After a possible shrinking of $U_k$ we assume that $\varphi$ is smooth on $U_k\backslash\{x_k\}$. Since 
    $\alpha_{E_j}(\widetilde{v_i})\in(-1,0]$
    for every $j=1,\cdots,s$, $e^{-\psi_1}, \dots, e^{-\psi_m}$ are locally integrable on $U_k$. This implies that
    $$\sI(\varphi+\psi_i)_x\simeq\sO_{U,x},\quad i=1,\dots,m,\quad \forall x\in U_k\backslash\{x_k\}.$$
    As a consequence,
    \begin{align}\label{align_siu_supp12}
    U\cap{\rm supp}\sO_X/\sI(\varphi+\psi_i)\subset\{x_k\},\quad i=1,\dots,m.
    \end{align}
   It is evident that $\nu(\varphi,x_k)\geq n+s_k$, which implies that
   \begin{align}\label{align_siu_supp32}
   \sI(\varphi+\psi_i)_{x_k}\subset \sI(\varphi)_{x_k}\subset m^{s_k+1}_{X,x_k},\quad i=1,\dots,m.
   \end{align}
    Combining (\ref{align_siu_cokernel2}), (\ref{align_siu_supp12}) and (\ref{align_siu_supp32}) we obtain a canonical surjective map
    $$H^0(X,S(IC_{X}(\bV))/S(IC_{X}(\bV),\varphi)\otimes A)\to \bigoplus_{1\leq k\leq r}S(IC_X(\bV))\otimes A\otimes \sO_{X,x_k}/m^{s_k+1}_{X,x_k}.$$
    Since (\ref{align_siu_surject12}) is surjective, we get the desired surjective map
    $$H^0(X,S(IC_X(\bV))\otimes A)\to \bigoplus_{1\leq k\leq r}S(IC_X(\bV))\otimes A\otimes \sO_{X,x_k}/m^{s_k+1}_{X,x_k}.$$
\end{proof}
Let $f:Y\to X$ be a proper holomorphic morphism from a K\"ahler manifold to a projective algebraic variety. Assume that the degenerate loci of $f$ is contained in a normal crossing divisor $D\subset X$. Denote by $\bV^q:=R^qf_\ast(\bC_{f^{-1}(X\backslash D)})$ the variation of Hodge structure on $X\backslash D$. Then $R^qf_\ast\omega_Y\simeq S(IC_X(\bV^q))$ is a locally free $\sO_X$-module for every $q\geq 0$ (\cite[Theorem 2.6]{Kollar1986} or \cite[Theorem V]{Takegoshi1995}). In this case, Corollary \ref{cor_Siu} and Corollary \ref{cor_Demailly} in \S \ref{section_geo_app} can be deduced.

When $L$ is ample and base point free, we could obtain the optimal bound.
\begin{thm}\label{thm_relFujita_semiample}
	Let $X$ be a projective algebraic variety and let $\bV$ be an $\bR$-polarized variation of Hodge structure on some Zariski open subset $X^o\subset X_{\rm reg}$. Let $L$ be an ample line bundle on $X$ which is generated by global sections and $G$ is a nef line bundle on $X$. Then $S(IC_X(\bV))\otimes L^{\otimes m}\otimes G$ generates simultaneous jets of order $s_1,\dots, s_r\in \bN$ at arbitrary points $x_1,\dots, x_r\in X$, i.e. there is a surjective map
	$$H^0(X,S(IC_X(\bV))\otimes L^{\otimes m}\otimes G)\to \bigoplus_{1\leq k\leq r}S(IC_X(\bV))\otimes L^{\otimes m}\otimes G\otimes \sO_{X,x_k}/m^{s_k+1}_{X,x_k},$$
	provided that $m\geq\dim_\bC X+\sum_{1\leq k\leq r}(s_k+1)$.
\end{thm}
\begin{proof}
By Theorem \ref{thm_Nadel_vanishing}, it follows that $S(IC_X(\bV))\otimes L^{\otimes(\dim_\bC X+1)}\otimes G$ is $0$-regular in the sense of Castelnuovo-Mumford \cite{Mumford1966}. Consequently, the theorem follows from \cite[Theorem 1.1]{Shentu2020}.
\end{proof}

\bibliographystyle{plain}
\bibliography{CGM}

\end{document}